\documentclass[bj,preprint]{imsart}

\RequirePackage[OT1]{fontenc}
\RequirePackage{amsthm,amsmath}
\RequirePackage[numbers]{natbib}

\usepackage{xifthen}
\usepackage{subfig,mathrsfs}
\newtheorem{example}{Example}

\newcommand{\rrVert}{\Vert}
\newcommand{\llVert}{\Vert}

\usepackage{pkgfile}

\allowdisplaybreaks

\usepackage{colortbl,tabu,xcolor}

\definecolor{pastelred}{rgb}{1.0, 0.74, 0.85}
\definecolor{electriccyan}{rgb}{0.0, 1.0, 1.0}

\begin{document}

\begin{frontmatter}
\title{Inference in Ising Models}
\runtitle{Inference in Ising Models}
%\thankstext{T1}{Footnote to the title with the ``thankstext'' command.}

\begin{aug}
\author{\fnms{Bhaswar~B.} \snm{Bhattacharya}\ead[label=e1]{bhaswar@wharton.upenn.edu}}
\and
\author{\fnms{Sumit} \snm{Mukherjee}\ead[label=e2]{sm3949@columbia.edu}}
\runauthor{Bhattacharya and Mukherjee}

\affiliation{University of Pennsylvania and Columbia University}

\address{Department of Statistics\\ 
University of Pennsylvania\\ Philadelphia, USA\\ \printead{e1}}

\address{Department of Statistics\\ Columbia University\\ New York, USA\\ \printead{e2}}
\end{aug}

\begin{abstract}
The Ising spin glass is a one-parameter exponential family model for
binary data with quadratic sufficient statistic. In this paper, we show
that given a single realization from this model, the maximum
pseudolikelihood estimate (MPLE) of the natural parameter is $\sqrt
{a_N}$-consistent at a point whenever the log-partition function has
order $a_N$ in a neighborhood of that point. This gives consistency
rates of the MPLE for ferromagnetic Ising models on general weighted
graphs in all regimes, extending the results of Chatterjee
(\textit{Ann. Statist.} \textbf{35} (2007) 1931--1946) where only
$\sqrt N$-consistency of the MPLE was shown. It
is also shown that consistent testing, and hence estimation, is
impossible in the high temperature phase in ferromagnetic Ising models
on a converging sequence of simple graphs, which include the
Curie--Weiss model. In this regime, the sufficient statistic is
distributed as a weighted sum of independent $\chi^2_1$ random
variables, and the asymptotic power of the most powerful test is
determined. We also illustrate applications of our results on synthetic
and real-world network data.
\end{abstract}

% KEYWORDS
%
\end{frontmatter}

%s1 #&#
\section{Introduction}\label{sec1}

The Ising spin glass is a discrete random field developed in
statistical physics
as a model for ferromagnetism \cite{ising}, and is now widely used in
statistics as a model for binary data with applications in spatial
modeling, image processing, and neural networks (cf. \cite
{spatialdata_book,green,hopfield} and the references therein). To
describe the model, suppose that the data is a vector of dependent $\pm
1$ random variables $\sigma=(\sigma_1, \sigma_2,\ldots, \sigma
_N)$, and the dependence among the coordinates of $\sigma$ is modeled
by a one-parameter exponential family where the sufficient statistic is
a quadratic form:
%
%e1.1 #&#
%
\begin{equation}
H_N(\tau)=\tau'J_N\tau= \sum
_{1\leq i, j\leq N}J_N(i, j)\tau_i
\tau_j \label{q}
\end{equation}
for any $\tau\in S_N:=\{-1, 1\}^N$ and an $N\times N$ symmetric matrix
$J_N$ with zeros on the diagonals. The elements of $J_N$ are denoted by
$J_N(i, j)=J_N(j, i)$, for $1\leq i< j\leq N$. Given any $\beta\geq
0$, the quadratic form (\ref{q}) defines a parametric family of
probability distributions on $S_N$:
%
%e1.2 #&#
%
\begin{equation}
\P_{\beta}(\sigma=\tau)=2^{-N} \exp\biggl\{\frac{1}{2}
\beta H_N(\tau)-F_N(\beta) \biggr\}, \label{pmf}
\end{equation}
where $F_N(\beta)$ is the \textit{log-partition function} which is
determined by the condition $\sum_{\tau\in S_N}\P_\beta\{\sigma
=\tau\}=1$, that is,
%
%e1.3 #&#
%
\begin{equation}
F_N(\beta):=\log\biggl\{\frac{1}{2^N}\sum
_{\tau\in S_N}e^{\frac
{1}{2}\beta H_N(\tau)} \biggr\}=\log\E_0
e^{\frac{1}{2}\beta
H_N(\sigma)}, \label{Z}
\end{equation}
where $\E_0$ denotes the expectation over $\sigma$ distributed as $\P
_0$, the uniform measure on $S_N$. The parameter $\beta=1/T$ is often
referred to as the \textit{inverse temperature}, so the high temperature
regime corresponds to small values of $\beta$. The family (\ref{pmf})
includes many famous statistical physics models: the usual
ferromagnetic Ising model on generals graphs, the
Sherrington--Kirkpatrick mean-field model \cite
{panchenko,talagrand,talagrandII}, and the Hopfield model for neural
networks \cite{hopfield}.

Estimating the parameter $\beta$ in \eqref{pmf}, given one
realization from the model, is extremely difficult using
likelihood-based methods because of the presence of an intractable
normalizing constant $F_N(\beta)$ in the likelihood. A variety of
numerical methods are known for approximately computing the likelihood
\cite{geyer_thompson}, but they are computationally expensive and very
little is known about the rate of convergence.

One alternative to using likelihood-based methods is to consider the
maximum pseudolikelihood estimator (MPLE) \cite
{besag_lattice,besag_nl}. Chatterjee \cite{Chatterjee} showed that
given a single spin configuration from the model \eqref{pmf}, the MPLE
$\hat\beta_N$ is $\sqrt N$-consistent at $\beta=\beta_0$,\footnote
{A sequence of estimators $\{\hat{\beta}_N\}_{N\geq1}$ is said to be
$a_N$-\textit{consistent} at $\beta=\beta_0$ if $a_N|\hat{\beta
}_N-\beta_0|=O_P(1)$, that is, $\limsup_{K\rightarrow\infty}\limsup
_{N\rightarrow\infty}\P_{\beta_0}(a_N|\hat{\beta}_N-\beta
_0|>K)=0$.} whenever $\liminf_{N\rightarrow\infty}\frac
{1}{N}F_N(\beta_0)>0$. However, in many popular models such as regular
graphs, random graphs, and dense graphs, the log-partition
function $F_N(\beta)=o(N)$ for certain ranges of $\beta$, and
Chatterjee's result does not tell us anything about the consistency of
the MPLE.

In this paper, we show that the MPLE is $\sqrt{a_N}$-consistent at
$\beta=\beta_0$, if the log-partition function has order $a_N$ in a
neighborhood of $\beta_0$ (Theorem~\ref{GEN}), for a sequence
$a_N\rightarrow\infty$. This gives the consistency rate of the MPLE
for all values of $\beta>0$ away from the critical points, and shows
that the rate of the MPLE undergoes phase transitions for Ising models
on various graphs ensembles (Corollaries \ref{dregular} and \ref
{erdosrenyirandom}). We also show that \textit{no} consistent test, and
hence no estimator, exists if the log-partition function remains
bounded (Theorem~\ref{genfixed}). As a consequence, consistent
estimation is impossible in the high temperature regime in
ferromagnetic Ising models on a converging sequence (in cut-metric as
defined by Lova\'sz and co-authors \cite
{graph_limits_I,graph_limits_II,lovasz_book}) of graphs
(Theorem~\ref{dense_testing}). This strengthens previous results of
Comets and Gidas \cite{comets_mle} and Chatterjee \cite{Chatterjee}
where the MLE and the MPLE was, respectively, shown to be inconsistent
for $0\leq\beta< 1$ in the Curie--Weiss model, which corresponds to
taking $J_N(i, j)=1/N$, for all $1\leq i < j \leq N$. Finally, using
the emerging theory of graph limits \cite
{graph_limits_I,graph_limits_II,lovasz_book}, the limiting distribution
of the sufficient statistic $H_N(\sigma)$, and the asymptotic power of
the most powerful test are derived for dense graphs in the
high temperature regime (Theorem~\ref{dense}).

%%%%%%%%%%%%%%%%%%%%%%%%%%%%%%%%%%%%%%%%%%%%%%%%%%%%%%%%%%%%%%%%%%%%%%%%%%%%%%%%%
%\begin{figure*}
%%
%\centering
%\begin{minipage}[l]{1.0\textwidth}
%\centering
%\includegraphics[width=3.5in]
% {0_Facebook.pdf}
%\end{minipage}
%
%\caption{\footnotesize{A Facebook friendship-network with 333 nodes
%and 2519 edges, where the vertices are colored red or blue according
%to gender. The MPLE of $\beta$ for the Ising model on this graph is
%0.853. The analysis of this dataset is given in Section~
%\ref{sec:datasets}}.}
%\label{fig:fb1}
%\end{figure*}
%%%%%%%%%%%%%%%%%%%%%%%%%%%%%%%%%%%%%%%%%%%%%%%%%%%%%%%%%%%%%%%%%%%%%%%%%%%%%%%%%%%%%%%%%%%%%%%%%%%%%%%%%%%

While proving the consistency of the MPLE, we show that the asymptotic
order of the sufficient statistic $H_N(\sigma)$ is same as the order
of the log-partition function for general matrices $J_N$; a result
which appears to be new and might be of independent interest. More
precisely, the sequence of random variables $\frac{1}{a_N}H_N(\sigma
)$ is asymptotically tight under $\P_{\beta_0}$, and the limiting
distribution (if any) is non-zero when the log-partition function has
order $a_N$ in a neighborhood of $\beta_0$ (Lemma~\ref{lem1}).
Moreover, simple bounds for matrices $J_N$ with non-negative entries
provide the correct order of the log-partition function in the high
temperature regime for a wide class of Ising models (Lemma~\ref{bound}).

Finally, we illustrate the usefulness of the MPLE and the applicability
of our results on a real dataset: In Section~\ref{sec:datasets}, we
study the effect of gender among friends in two Facebook
friendship-networks from the Stanford Large Network Dataset (SNAP) collection.

%%%%%%%%%%%%%%%%%%%%%%%%%%%%%%%%%%%%%%%%%%%%%%%%%%%%%%%%%%%%%%%%%%%%%%%%%%%%%%%%%
%\begin{figure*}
%%
%\centering
%\begin{minipage}[c]{1.0\textwidth}
%\centering
%\includegraphics[height=3.0in, width=3.5in]
% {2012_Graph.pdf}
%
%\end{minipage}
%\caption{\footnotesize{Ising model on the US neighborhood graph for
%the 2012 presidential election: the vertices are the states and there
%is an edge between two states which share a border. The color of a
%vertex is blue or red depending on whether the state was Democratic or
%Republican in the 2012 presidential elections, respectively. The MPLE
%of the Ising model on this graph is $0.9611$. The analysis of this
%dataset is given in Section~\ref{sec:statesdataset}.}}
%\label{fig:states2012}
%\end{figure*}
%%%%%%%%%%%%%%%%%%%%%%%%%%%%%%%%%%%%%%%%%%%%%%%%%%%%%%%%%%%%%%%%%%%%%%%%%%%%%%%%%%%%%%%%%%%%%%%%%%%%%%%%%%%

Another active area of research is high-dimensional structure
estimation in a sparse Ising model, where the goal is to consistently
estimate the underlying matrix $J_N$, under certain structural
constraints from i.i.d. samples from the model (see \cite
{structure_learning,bresler,highdim_ising,ising_nonconcave} and the
references therein). This is in contrast with the present work, where
the matrix $J_N$ is known and we estimate the natural parameter and its
error rate given a \textit{single} realization from the model.

%s1.1 #&#
\subsection{Organization} The rest of the paper is organized as
follows: The consistency of the MPLE and general inconsistency results
are described in Section~\ref{pl_estimate}. Applications of these
results to various graph ensembles including regular graphs, random
graphs, and general weighted graphs, are explained in Section~\ref
{examples}. Theorems \ref{GEN} and \ref{genfixed} are proved
in Sections \ref{pfmain} and \ref{pfgenfixed}, respectively.
The proofs of Corollaries \ref{dregular} and \ref{erdosrenyirandom} are
given in Section~\ref{pfapplication}. The results on converging
sequence of graphs are in Section~\ref{pfdense}. The analysis
of the Facebook dataset is given in Section~\ref{sec:datasets}.

%s2 #&#
\section{Consistency of the MPLE}
\label{pl_estimate}

The maximum pseudolikelihood estimator (MPLE), introduced by Besag
\cite{besag_lattice,besag_nl}, can be conveniently used to approximate
the joint distribution of $\sigma\sim\P_\beta$ that avoids
calculations with the normalizing constant.

%de2.1 #&#
%
\begin{defn}Given a random vector $(X_1, X_2, \ldots, X_N)$ whose
joint distribution is parame\-trized by a parameter $\beta\in\R$, the
MPLE of $\beta$ is defined as
%
%e2.1 #&#
%
\begin{equation}
\hat\beta_N:=\arg\max\prod_{i=1}^N
f_i(\beta, X), \label{mple}
\end{equation}
where $f_i(\beta, X)$ is the conditional probability density of $X_i$
given $(X_j)_{j \ne i}$.
\end{defn}

Given $\sigma\sim\P_\beta$ from the model (\ref{pmf}), the
conditional density of $\sigma_i$, given $(\sigma_j)_{j \ne i}$ can
be easily computed. To this end, given $\tau\in S_N$, define the
function $L_\tau:[0, \infty)\rightarrow\R$ as
%
%e2.2 #&#
%
\begin{equation}
L_\tau(x):=\frac{1}{N}\sum_{i=1}^N
m_i(\tau) \bigl(\tau_i- \tanh\bigl(x m_i(\tau)
\bigr)\bigr), \label{s}
\end{equation}
where
%
%e2.3 #&#
%
\begin{equation}
m_i(\tau):=\sum_{j=1}^NJ_N(i,j)
\tau_j. \label{mi}
\end{equation}
Note that $m_i(\tau)$ does not depend on $\tau_i$ since the diagonal
element $J_N(i, i)=0$. Interpreting $\tanh(\pm\infty)=\pm1$, the
function $L_{\tau}$ can be extended to $[0, \infty]$ by defining
$L_\tau(\infty):=\frac{1}{N}\sum_{i=1}^N(m_i(\tau)\tau
_i-|m_i(\tau)|)$. Then it is easy to verify that (see Chatterjee \cite
{Chatterjee}, Section~1.2) $\frac{1}{N}\frac{\partial}{\partial
\beta}\sum_{i=1}^N \log f_i(\beta, \tau)= L_{\tau}(\beta)$, and the
function $L_{\tau}(\beta)$ is a decreasing function of $\beta$.
Therefore, the MPLE for $\beta$ in the model (\ref{pmf}) is
%
%e2.4 #&#
%
\begin{equation}
\hat\beta_N(\sigma):=\inf\bigl\{x\geq0: L_\sigma(x)=0\bigr
\}, \label{beta}
\end{equation}
where $\sigma\sim\P_\beta$ is a random element from (\ref{pmf}).
Hereafter, we suppress the dependence on $\sigma$ and denote by $\hat
\beta_N:=\hat\beta_N(\sigma)$ the MPLE of $\beta$.
%
%In this paper we stu\mathrm d y the consistency properties of the MPLE
%estimator $\beta_N$. To this end, we have the following definition:
%
%\begin{defn}
%\label{defn:const}
%\label{consistency}
%\end{equation}
%\end{defn}

Consistency results for the MPLE in Ising models are known in the case
of lattices \cite{comets_exp,gidas,guyon,pickard}, complete graphs
\cite{Chatterjee}, and spatial point processes \cite{jensen_moller}.
However, for general processes where the dependence is neither local
nor mean-field, it is very difficult to prove consistency results for MPLE.
%even though it often performs well in practice.
In a major breakthrough, Chatterjee \cite{Chatterjee} developed a
remarkable technique using exchangeable pairs and showed \cite
{Chatterjee}, Theorem~1.1, that the MPLE $\{\hat{\beta}_N\}_{N\geq
1}$, given a
single realization $\sigma\in S_N$ from (\ref{pmf}), is a $\sqrt
N$-consistent estimator at $\beta=\beta_0>0$, whenever $\sup
_N\llVert J_N\rrVert <\infty$\footnote{For any $N\times N$
symmetric matrix $A$,
denote by $\llVert A\rrVert =\sup_{x\in\R^N}\frac{\llVert A
x\rrVert _2}{\llVert x\rrVert _2}$ the
operator norm of $A$.} and
%
%e2.5 #&#
%
\begin{equation}
\lim\inf_{N\rightarrow\infty}\frac{1}{N}F_N(
\beta_0)>0. \label{ptN}
\end{equation}

%%%%%%%%%%%%%%%%%%%%%%%%%%%%%%%%%%%%%%%%%%%%%%%%%%%%%%%%%%%%%%%%%%%%%%%%%%%%%%%%%

%
%\begin{figure*}[h]
%\centering
%\begin{minipage}[c]{1.0\textwidth}
%\centering
%\includegraphics[width=2.8in,height=2.8in]
% {MPLE_I.pdf}
%\end{minipage}
%%\begin{minipage}[c]{0.49\textwidth}
%%\centering
%%\includegraphics[width=3.0in]
%% {KNPowerL.pdf}
%%\end{minipage}
%
%\end{figure*}
%%%%%%%%%%%%%%%%%%%%%%%%%%%%%%%%%%%%%%%%%%%%%%%%%%%%%%%%%%%%%%%%%%%%%%%%%%%%%%%%%%%%%%%%%%%%%%%%%%%%%%%%%%%

To the best of our knowledge, all results regarding MPLE $\{\hat{\beta
}_N\}_{N\geq1}$ are in the regime where it is $\sqrt N$-consistent.
However, in many examples such as the Ising model on dense graphs,
$d(N)$-regular graphs with $d(N)\rightarrow\infty$, and
Erd\H os--R\'enyi graphs $G(N, p(N))$, with $\frac{\log N}{N}\ll p(N) \ll1$, the
log-partition function $F_N(\beta)=o(N)$ for certain ranges for $\beta
$. In these cases, the hypothesis \eqref{ptN} is not satisfied, and
Chatterjee's result is not applicable for deriving the consistency of
the MPLE. The following theorem (see Section~\ref{pfgen} for proof)
shows that the consistency of the MPLE at a point is governed by the
order of the log-partition function in a neighborhood of that point.
This generalizes the result of Chatterjee \cite{Chatterjee} giving the
rate of consistency of the MPLE for all values $\beta$ (at all
temperatures) away from the critical points.

%th2.1 #&#
%
\begin{thm}\label{GEN}
Let $\sup_{N\ge1}\llVert J_N\rrVert <\infty$, and $\beta_0>0$ be
fixed. Suppose
$\{a_N\}_{N\ge1}$ is a sequence of positive reals diverging to $\infty
$ such that for some $\delta>0$ we have
%
%e2.6 #&#
%
\begin{align}
0<\liminf_{N\rightarrow\infty}\frac{1}{a_N}F_N(
\beta_0-\delta)\le\limsup_{N\rightarrow\infty}
\frac{1}{a_N}F_N(\beta_0+\delta)<\infty.
\label{maincond}
\end{align}
Moreover, assume that the following conditions hold:
\begin{enumerate}[(a)]

\item[(a)]$\limsup_{K\rightarrow\infty}\limsup_{N\rightarrow
\infty}\frac{1}{a_N}\E_{\beta_0} (\sum_{i=1}^N |m_i(\sigma
)|\cdot\pmb1\{|m_i(\sigma)|>K\} )=0$, where $m_i(\sigma)$ is
as defined in (\ref{mi}).

\item[(b)]$\limsup_{N\rightarrow\infty}\frac{1}{a_N}\sum
_{i,j=1}^N J_N(i, j)^2<\infty$.
\end{enumerate}
Then the MPLE $\{\hat{\beta}_N\}_{N\geq1}$ for the model (\ref
{pmf}) is a $\sqrt{a_N}$-consistent sequence of estimators for $\beta
=\beta_0$.
\end{thm}

Conditions (a) and (b) are technical requirements arising out of the
proof technique, which ensure that the main contributions come from
$m_i(\sigma)$ that are small, and on average the entries in $J_N$ are
not too large compared to $a_N$. The proof of the result is given in
Section~\ref{pftechnicallemmas} (with technical lemmas proved in
Appendix \ref{pfgenlemmas}). The proof is organized as follows: Using
the two conditions of the theorem, Lemma~\ref{2moment} shows that $\E
_{\beta_0} (L_\sigma(\beta_0)^2)=O(a_N/N^2)$, which implies that
$L_\sigma(\beta_0)$ is small with high probability. To derive the
rate of consistency of the pseudo-likelihood, it thus suffices to get a
lower bound of the derivative $L'_\sigma(\beta)$. Again invoking the
two conditions of Theorem~\ref{GEN}, in Lemma~\ref{pflemderivative}
we derive a lower bound on
\[
\sum_{i=1}^N m_i(
\sigma)^2 \pmb1\bigl\{\bigl|m_i(\sigma)\bigr| \leq K\bigr\}
\]
for $K$ fixed. This translates into the desired lower bound on the
derivative $L'_\sigma(\beta)$ using which the proof of the theorem is
then completed.

The conditions of the theorem are satisfied in most commonly used
models (see Section~\ref{examples}). Moreover, the result of
Chatterjee \cite{Chatterjee}, Theorem~1.1, is an immediate corollary of
Theorem~\ref{GEN} (refer to Section~\ref{pfcorchatterjee} for the proof).

%co2.2 #&#
%
\begin{cor}[({\cite{Chatterjee}, Theorem~1.1})]
\label{sc}
Let $\sup_{N\ge1}\llVert J_N\rrVert <\infty$ and $\beta_0>0$ be
such that (\ref
{ptN}) holds. Then the sequence of estimators $\{\hat{\beta}_N\}
_{N\geq1}$ is $\sqrt{N}$ consistent for $\beta=\beta_0$.
\end{cor}

%re2.1 #&#
%
\begin{remark}Condition (\ref{maincond}) in the Theorem~\ref{GEN}
demands the right order of the log-partition function in a small
neighborhood around the point $\beta_0$. This avoids the critical
points, where the order of the log-partition function (and its
derivative) undergoes a sharp transition. It follows from the proof of
Theorem~\ref{GEN} that the following (possibly slightly weaker)
condition works as well instead of \eqref{maincond}:
\[
0<\lim_{\delta\rightarrow0}\liminf_{N\rightarrow\infty}
\frac
{1}{a_N}F_N'(\beta_0-\delta)\le
\limsup_{N\rightarrow\infty}\frac
{1}{a_N}F_N'(
\beta_0)<\infty.
\]
However, for most of the applications estimates of the log-partition
function are more readily available. Thus, the sufficient conditions
are stated in terms of the log-partition function instead of its derivative.
\end{remark}

Note that Theorem~\ref{GEN} does not apply to the case $F_N(\beta
_0)=O(1)$. Next, we show that if $F_N(\beta_0)=O(1)$, then there is no
sequence of estimators which consistently estimates $\beta_0$. In
fact, we show that even testing is impossible in this regime: Given a
single spin-configuration $\sigma\in S_N$ from (\ref{pmf}), there
exists \textit{no} sequence of consistent tests\footnote{A sequence of
test functions $\phi_N: S_N\rightarrow\{0, 1\}$ is said to be {\it
consistent} for the testing problem (\ref{testingbeta12}) if $\lim
_{N\rightarrow\infty}\E_{\beta_1}\phi_N=0$ and $\lim
_{N\rightarrow\infty}\E_{\beta_2}\phi_N=1$.
} for the hypothesis testing problem:
%
%e2.7 #&#
%
\begin{equation}
H_0: \beta=\beta_1 \quad\text{versus}\quad H_1:
\beta=\beta_2. \label{testingbeta12}
\end{equation}
This is summarized in the following theorem (see Section~\ref
{pfgenfixed} for proof):
%
%th2.3 #&#
%
\begin{thm}\label{genfixed}
Let $\sup_{N\ge1}\llVert J_N\rrVert <\infty$, and $\beta_0>0$ be
fixed. Suppose
%
%e2.8 #&#
%
\begin{align}
\limsup_{N\rightarrow\infty} F_N(\beta_0)<\infty.
\label{maincondfixed}
\end{align}
Then for $0\leq\beta_1<\beta_2\leq\beta_0$, there exists no
consistent sequence of tests for the testing problem~(\ref
{testingbeta12}). In particular, there exists no consistent sequence of
estimators for $\beta$ in the interval $[0,\beta_0]$.
\end{thm}

%This theorem will be used later to show that for Ising models on dense
%graphs no consistent testing and estimation is possible in the high
%temperature regime (Theorem~\ref{dense}).

One of the main applications of above results is in deriving the rate
of the MPLE for Ising models on weighted graphs, that is, for matrices
$J_N$ with non-negative entries. For such matrices, condition (b) in
Theorem~\ref{GEN} can be directly verified, and we have the following
simplified corollary:

%co2.4 #&#
%
\begin{cor}\label{mplepositive}
Consider the model (\ref{pmf}) such that $J_N$ is a sequence of
matrices with non-negative entries with $\lim_{N\rightarrow\infty
}\llVert J_N\rrVert =\lambda>0$.
\begin{enumerate}[(a)]
\item[(a)] The sequence of estimators $\{\hat{\beta}_N\}_{N\geq1}$ is
$\llVert J_N\rrVert _F:=\sqrt{\sum_{i, j=1}^N J_N(i, j)^2}$
consistent at $\beta
=\beta_0$ for any $\beta_0<\frac{1}{\lambda}$, whenever condition
\textup{(a)} in Theorem~\ref{GEN} holds.

\item[(b)] If
$\limsup_{N\rightarrow\infty}\sum_{i, j=1}^N J_N(i, j)^2< \infty$,
then exists no consistent sequence of estimators for $\beta$ in the
interval $[0,\frac{1}{\lambda})$.
\end{enumerate}
\end{cor}

%s3 #&#
\section{Applications}
\label{examples}

The $\sqrt N$-consistency of the MPLE in the Sherrington--Kirkpatrick
(SK) model and the Hopfield model, for all values of $\beta> 0$,
follows from results of Chatterjee \cite{Chatterjee}. Our results give
the rate of consistency of the MPLE in the regime where it is not
$\sqrt N$-consistent.

We begin with a simple example where the rate of the MPLE undergoes
multiple phase transitions.

%ex1 #&#
%
\begin{example}\label{ptblock} Consider the model \eqref{pmf} with
\begin{align*}
J_N(i,j)=
\begin{cases} \frac{1}{N},
& \mbox{if } 1\le i \ne j\le\frac{N}{2},
\\
\frac{1}{\sqrt{N}}, & \text{if }  \frac{N}{2}<i\ne j\le
\frac
{N}{2}+\sqrt{N},
\\
0, & \text{otherwise.}
\end{cases}
\end{align*}
Then the sequence of estimators $\{\hat{\beta}_N\}_{N\geq1}$ is
inconsistent for $\beta\in(0,1)$, $N^{1/4}$-consistent for $\beta\in
(1,2)$, and $\sqrt{N}$-consistent if $\beta>2$.
\end{example}

The proof of the above example is given in Section~\ref{ptblockpf}. In
fact, this example can be easily generalized to construct a $K$-block
matrix $J_N$ such that the consistency rate of MPLE undergoes $K$ phase
transitions. However, for most popular choices of $J_N$ the rate of the
MPLE undergoes at most one phase transition from $\llVert J_N\rrVert
_F$-consistent
to $\sqrt N$-consistent.

%s3.1 #&#
\subsection{Ising model on regular graphs}
Let $G_N$ be a sequence of $d_N$ regular graphs. Consider the family of
probability distributions (\ref{pmf}) with the sufficient statistic
%
%e3.1 #&#
%
\begin{equation}
H_N(\tau)=\frac{1}{d_N}\tau'A(G_N)
\tau, \label{adjm}
\end{equation}
where $A(G_N)=((a_N(i, j)))$ is the adjacency matrix of the graph
$G_N$. This includes Ising models on lattices, complete graph,
hypercube, and random regular graphs, among others, and have been
extensively studied in probability and statistical physics. Dembo et
al. \cite{dembom_sun,dembom} derived the limit of the log-partition
function for random regular (and other locally-tree like) graphs.
%The Glauber dynamics is one of the most practiced methods to sample
%the Gibbs distribution, and has been extensively studied for regular
%graphs.
Levin \textit{et al.} \cite{peres_complete_graph} showed that the
mixing time
of the Glauber dynamics on the complete graph exhibits the cutoff
phenomenon \cite{diaconis_cutoff} in the high temperature regime. The
cutoff phenomenon for lattices was established by Lubetzky and Sly in a
series of breakthrough papers (refer to \cite
{lubetzky_sly_ip,lubetzky_sly_lattice} and the references therein).

The next result gives the rate of consistency of the MPLE for general
regular graphs. The proofs are deferred to Section~\ref{pfapplication}.

%co3.1 #&#
%
\begin{cor}\label{dregular}
Fix $\beta_0>0$ and let $G_N$ be a sequence of $d_N$ regular graphs.
Suppose $\{\hat{\beta}_N\}_{N\geq1}$ is the MPLE for the model (\ref
{pmf}) with sufficient statistic (\ref{adjm}).
\begin{enumerate}[(a)]

\item[(a)] If $0<\beta_0<1$, $\{\hat{\beta}_N\}_{N\geq1}$ is a $\sqrt
{N/d_N}$-consistent sequence of estimators for $\beta_0$.

\item[(b)] If $\beta_0>1$, $\{\hat{\beta}_N\}_{N\geq1}$ is a $\sqrt
{N}$-consistent sequence of estimators for $\beta_0$.
\end{enumerate}
\end{cor}

The above theorem shows that the rate of the MPLE undergoes a phase
transition at $\beta=1$ for general regular graphs. In particular if
$d_N=d=O(1)$ remains bounded, then the above theorem shows that the
MPLE is $\sqrt N$ for all non-negative $\beta\ne1$. However, in this
case, it is easy to argue that $\liminf_{N\rightarrow\infty}\frac
{1}{N}F_N(\beta)>0$, for all $\beta> 0$ (see proof of lower bound in
Corollary~\ref{dregular}). Theorem~\ref{GEN} then concludes that
$\hat\beta_N$ is $\sqrt{N}$-consistent for all values of $\beta>
0$. In fact, using similar arguments as in the proof of Corollary~\ref
{dregular}, it follows that the MPLE $\{\hat{\beta}_N\}_{N\geq1}$ is
$\sqrt N$-consistent for all $\beta>0$ in all bounded degree graphs
with at least $O(N)$ edges. This shows that MPLE is $\sqrt
N$-consistent for lattice graphs re-deriving classical results (see
\cite{guyon} and the references therein).

For $d_N\rightarrow\infty$, the behavior of the MPLE at $\beta=1$
remains unclear. It is believed that the MPLE might have a non-Gaussian
limiting distribution at the critical point $\beta=1$ \cite{Chatterjee}.

%re3.1 #&#
%
\begin{remark}
If $d_N=\Theta(N)$,\footnote{Given non-negative sequences $\{a_N\}
_{N\geq1}$ and $\{b_N\}_{N\geq1}$, the notation $a_N=\Theta(b_N)$
means that there exist constants $k_1, k_2>0$, such that $k_1b_N\leq
a_N \leq k_2 b_N$, for all $N$ large enough.} then Theorem~\ref
{dregular} shows that the MPLE is $O(1)$ consistent for $0<\beta_0<1$,
suggesting that the MPLE might be inconsistent in this regime.
Chatterjee \cite{Chatterjee} showed that this is indeed the case for
the Curie--Weiss model (where $J_N(i, j)=1/N$ for all $i\ne j$) for
$0\leq\beta<1$. Comets and Gidas \cite{comets_mle} showed that even
the MLE of $\beta$ in the Curie--Weiss model is inconsistent for
$0\leq
\beta<1$. Later, in Theorem~\ref{dense} we strengthen this result by
showing that for Ising models on arbitrary dense graphs, there exists
no sequence of consistent estimators before the phase transition point.
This extends the results in \cite{Chatterjee,comets_mle} and justifies
the $O(1)$-rate of the MPLE in the dense case.
\end{remark}

%%%%%%%%%%%%%%%%%%%%%%%%%%%%%%%%%%%%%%%%%%%%%%%%%%%%%%%%%%%%%%%%%%%%%%%%%%%%%%%%%%
%\begin{figure*}[h]
%\centering
%\begin{minipage}[c]{1.0\textwidth}
%\centering
%\includegraphics[width=3.95in, height=3.35in]
% {MPLE_II.pdf}
%\end{minipage}
%%\begin{minipage}[c]{0.49\textwidth}
%%\centering
%%\includegraphics[width=3.0in]
%% {KNPowerL.pdf}
%%\end{minipage}
%\caption{\footnotesize{Histograms of the (scaled) MPLE in an Ising
%model on $G_N\sim\cG(N, p(N))$ with $N=2000$ and $p(N)=N^{-
%\frac{1}{3}}$ over 100 repetitions for $\beta=0.9<1$ and $
%\beta=1.1>1$. Note that at $\beta=0.9$, the MPLE scaled by $\sqrt N$
%(red histogram) has much larger variance compared to the correct
%scaling of $N^{\frac{1}{6}}$ (green histogram) as predicted by
%Corollary~\ref{erdosrenyirandom}. For $\beta=1.1$ the MPLE scales like
%$\sqrt N$ (blue histogram).}}
%\label{mplehist}
%\end{figure*}
%%%%%%%%%%%%%%%%%%%%%%%%%%%%%%%%%%%%%%%%%%%%%%%%%%%%%%%%%%%%%%%%%%%%%%%%%%%%%%%%%%%%%%%%%%%%%%%%%%%%%%%%%%%%
%

%s3.2 #&#
\subsection{Ising model on Erd\H{o}s--R\'enyi graphs}
Let $G_N\sim\cG(N, p(N))$ be a sequence of Erd\H os--R\'enyi graphs.
Consider the family of probability distributions (\ref{pmf}) with the
sufficient statistic
%
%e3.2 #&#
%
\begin{equation}
H_N(\tau)=\frac{1}{N p(N)}\tau'A(G_N)
\tau, \label{er}
\end{equation}
where $A(G_N)=((a_N(i, j)))$ is the adjacency matrix of the graph $G_N$.

%co3.2 #&#
%
\begin{cor}\label{erdosrenyirandom}
Fix $\beta_0>0$ and consider a sequence $G_N\sim\cG(N, p(N))$ of
Erd\H os--R\'enyi graphs, with $\frac{\log N}{N}\ll p(N) \leq1$. Let
$\{\hat{\beta}_N\}_{N\geq1}$ be the MPLE for the model (\ref{pmf})
with sufficient statistic~(\ref{er}).
%\footnote{For $p(N)=O(\frac{\log N}{N})$, the spectral norm of $
%\frac{1}{Np(N)}A(G_N)$ blows up \cite{ks}. Moreover, for $p(N)\leq
%\frac{(1-\varepsilon)\log N}{N}$ an Erd\H os-R\'enyi random graph has
%isolated vertices with high probability.}
%
\begin{enumerate}[(a)]

\item[(a)] If $0<\beta_0<1$, $\{\hat{\beta}_N\}_{N\geq1}$ is a $\sqrt
{1/p(N)}$-consistent sequence of estimators for $\beta_0$.

\item[(b)] If $\beta_0>1$, $\{\hat{\beta}_N\}_{N\geq1}$ is a $\sqrt
{N}$-consistent sequence of estimators for $\beta_0$.
\end{enumerate}
\end{cor}

As in the regular case, the rate of the MPLE undergoes a phase
transition at $\beta=1$ for Erd\H os--R\'enyi graphs. Figure~\ref
{mpleerrorrate} shows the error bars for the MPLE for the Ising model
on $G_N\sim\cG(N, p(N))$, with $N=2000$ and $p(N)=N^{-\frac{1}{3}}$,
for a sequence of values of $\beta\in[0, 2]$.

%
%f1 #&#
%

%%%%%%%%%%%%%%%%%%%%%%%%%%%%%%%%%%%%%%%%%%%%%%%%%%%%%%%%%%%%%%%%%%%%%%%%%%%%%%%
\begin{figure*}[h]
\centering
\begin{minipage}[c]{1.0\textwidth}
\centering
\includegraphics[width=2.8in,height=2.8in]
    {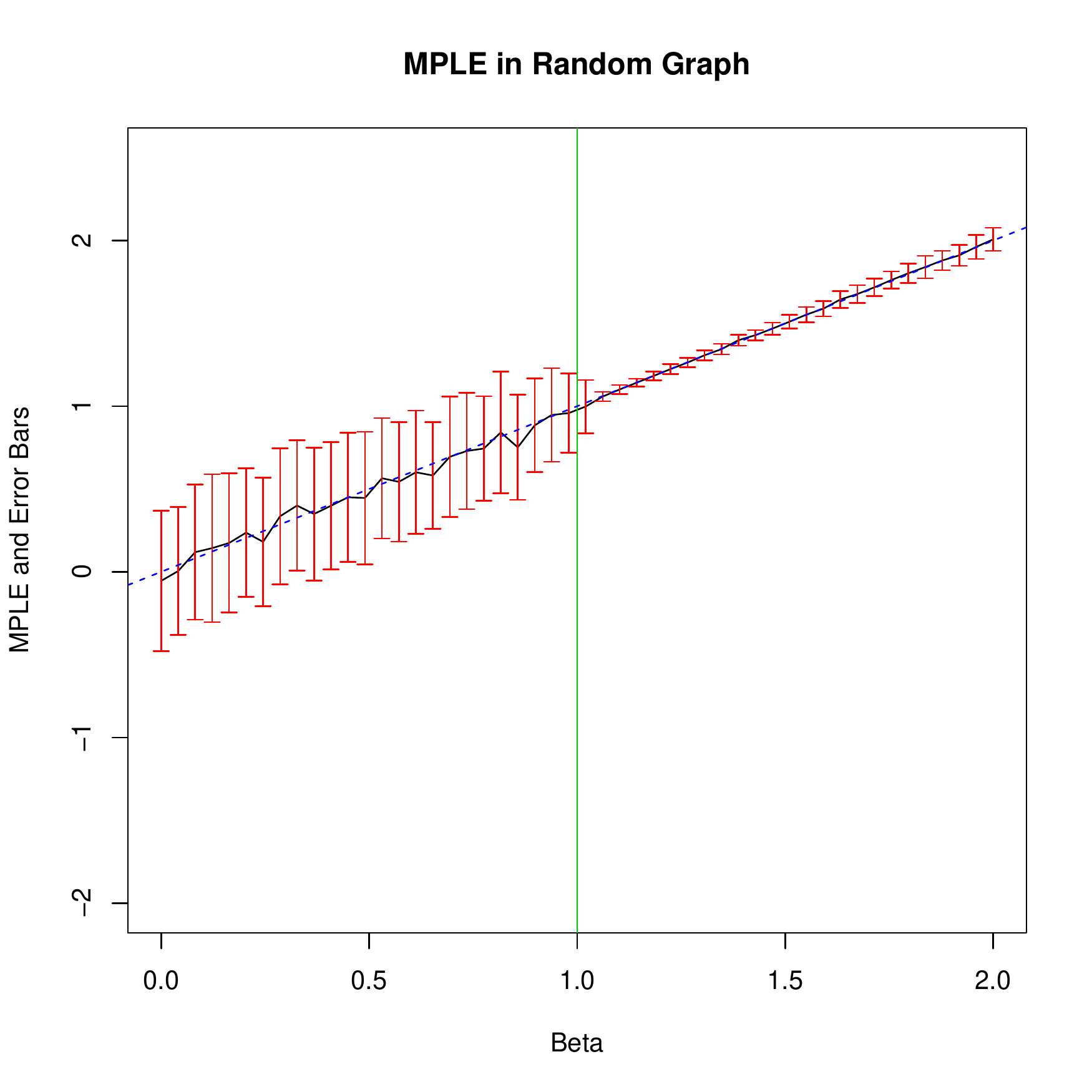}\\
\end{minipage}
%\begin{minipage}[c]{0.49\textwidth}
%\centering
%\includegraphics[width=3.0in]
%    {KNPowerL.pdf}\\
%\end{minipage}
\caption{\footnotesize{The MPLE and the 1-standard deviation error bar in an Ising model on $G_N\sim \cG(N, p(N))$ with $N=2000$ and $p(N)=N^{-\frac{1}{3}}$ averaged over 100 repetitions for a sequence of values of $\beta \in [0, 2]$. Lengths of the error bars undergo a phase transition at $\beta=1$, as predicted by Corollary~\ref{erdosrenyirandom} which shows that for $0\leq \beta<1$ the MPLE is $N^{\frac{1}{6}}$ consistent, and for $\beta>1$, the MPLE is $\sqrt N$-consistent.}}
\label{mpleerrorrate}\vspace{-0.2in}
\end{figure*}
%%%%%%%%%%%%%%%%%%%%%%%%%%%%%%%%%%%%%%%%%%%%%%%%%%%%%%%%%%%%%%%%%%%%%%%%%%%%%%%%%%%%%%%%%%%%%%%%%%%%%%%%%

%s3.3 #&#
\subsection{Ising model on dense graphs}

Recall that the MPLE is inconsistent in the Curie--Weiss model in the
high temperature regime, $0 \leq\beta< 1$ \cite{Chatterjee}. In this
section, using the emerging theory of graph limits and Theorem~\ref
{genfixed} above, we strengthen this result to show that consistent
testing is impossible in the entire high temperature regime in Ising
models on a converging sequence of dense graphs. We also
calculate the distribution of the most powerful test and the asymptotic
power in this regime.

%s3.3.1 #&#
\subsubsection{Graph limit theory} Let $G_N$ be a simple graph with vertices $V(G_N)=\{1, 2, \ldots, N\}$ and adjacency matrix
$A(G_N)$. %which is %=((w(a, b)))_{1\le a,b\leq N}$
%a symmetric matrix of
%edge-weights in $[0, 1]$ with zeros on the %diagonal, that is,
%
%e3.3 #&#
%
%\begin{equation}
%w(a, a)=0,\qquad w(a, b)\in[0, 1],\qquad w(a, b)=w(b, a)\qquad \text{for }1\leq a<
%b\leq N.
%\label{edgeweights}
%\end{equation}
%
%Conversely, any symmetric $N\times N$ matrix with entries in $[0, 1]$
%and zeros on the diagonal is the edge weight matrix of some weighted
%graph. A graph is said to be \textit{simple} if $w(a, a)=0$ and $w(a,
%b)=w(b,a)\in\{0, 1\}$, for $1\leq a<b\leq N$. In this case, $E(G_N)=\{
%w(a, b): w(a, b)=1, 1\leq a<b\leq N\}$ denotes the set of edges in the
%graph $G_N$.
Lov\'asz and co-authors \cite{graph_limits_I,graph_limits_II}
developed a limit theory of graphs, which connects various
topics such as graph homomorphisms, Szemer\'edi regularity lemma, and
extremal graph theory. In the following, we summarize the basic results for converging sequence of graphs  (cf. Lov\'asz \cite{lovasz_book} for
a detailed exposition).
To this end, note that any graph $G_N$ %, with edge-weights as in (\ref{edgeweights}),
can be represented as a function $W_{G_N}: [0,
1]^2\rightarrow[0, 1]$ in a natural way: Define\vadjust{\goodbreak} $W_{G_N}(x, y)
:=1$ if and only if $(\ceil{nx},\ceil{ny})$ is an edge in $G_N$, %w(\ceil{nx}, \ceil{ny})$.
that is, partition $[0, 1]^2$ into $N^2$
squares of side length $1/N$, and define $W_{G_N}(x, y)=1$, when $(x, y)$ is in the
$(a, b)$th square and $(a,b)$ is an edge in $G_N$. Let $\sW$ be the space of all measurable functions
from $[0, 1]^2$ into $[0, 1]$ that satisfy $W(x, y) = W(y,x)$ for all
$x, y\in[0, 1]$. For every $W\in\sW$ and any fixed simple graph
$H=(V(H), E(H))$ define the \textit{homomorphism density}
\[
t(H,W) =
\int_{[0,1]^{|V(H)|}}\prod_{(i,j)\in E(H)}W(x_i,x_j)
\,\mathrm{d} x_1\,\mathrm{d} x_2\cdots\,\mathrm{d}
x_{|V(H)|}.
\]
A sequence of simple graphs $\{G_N\}_{N\geq1}$ is said to converge
to $W\in\sW$ if for every finite simple graph $H$,
%
%e3.4 #&#
%
\begin{equation}
\lim_{N\rightarrow\infty}t(H, G_N) = t(H, W). \label{eq:graph_limit}
\end{equation}
The limit objects, that is, the elements of $\sW$, are called {\it
graph limits} or \textit{graphons}. Conversely, every such function arises
as the limit of an appropriate graph sequence.

It turns out that the above notion of convergence can be suitably
metrized using the so-called \textit{cut-metric} (cf. \cite
{lovasz_book}, Chapter~8, for details).
%$$d_\square(f,g):=\sup_{S,T\subset[0,1]}\left|\int_{S\times
%T}[f(x,y)-g(x,y)]\mathrm d x\mathrm d y\right|,$$ for $f, g\in\sW$.
%Define an equivalence relation on $\sW$ as follows: $f\sim g$ whenever
%$f(x, y)=g_\sigma(x, y):=g(\sigma x, \sigma y)$, for some measure
%preserving bijection $\sigma:[0,1]\mapsto[0,1] $. Denote by $
%\widetilde g$ the closure of the orbit $g_\sigma$ in $(\sW, d_
%\square)$. The quotient space is denoted by $\widetilde{\sW}$ and
%associated with the following natural metric: $$\delta_\square(
%\widetilde{f},\widetilde{g}):=\inf_{\sigma}d_\square(f,g_\sigma)=\inf_{
%\sigma}d_\square(f_\sigma,g)=\inf_{\sigma_1,\sigma_2}d(f_{\sigma_1},g_{
%\sigma_2}).$$ Then the space $(\widetilde{\sW},\delta_\square)$ is
%compact \cite{graph_limits_I}, and the metric $\delta_\square$ is
%commonly referred to as the \textit{cut-metric}. The main result in
%graph
%limit theory is that a sequence of weighted graphs $\{G_N\}_{N\geq1}$
%converges to a limit $W\in\sW$ in the sense defined in (
%\ref{eq:graph_limit}) if and only if $\delta_\square(\widetilde
%W_{G_N}, \widetilde W)\rightarrow0$
%\cite[Theorem~3.8]{graph_limits_I}. More generally, a sequence $\{
%\widetilde W_N\}_{n
%\geq1}$ converges to $\widetilde W\in\widetilde\sW$ if and only if $
%\lim_{N\rightarrow\infty}\delta_\square(\widetilde W_N, \widetilde
%W)=0$.
Moreover, every function $W\in\sW$ defines an operator $T_W: L_2[0,
1]\rightarrow L_2[0, 1]$:
%
%e3.5 #&#
%
\begin{eqnarray}
(T_Wf) (x)=
\int_0^1W(x, y)f(y)\,\mathrm{d} y. \label{eq:T}
\end{eqnarray}
$T_W$ is a Hilbert--Schmidt operator with operator norm denoted by
$\llVert W\rrVert $, which is compact and has a discrete spectrum,
that is, a
countable multiset of non-zero real eigenvalues $\{\lambda_i(W)\}_{i
\in\N}$. In particular, every non-zero eigenvalue has finite
multiplicity and
%
%e3.6 #&#
%
\begin{eqnarray}
\sum_{i=1}^\infty\lambda_i^2(W)=
\int_{[0, 1]^2}W(x, y)^2\,\mathrm{d} x\,\mathrm{d} y:=\llVert
W\rrVert_2^2. \label{lambda2}
\end{eqnarray}

%s3.3.2 #&#
\subsubsection{Consistency and asymptotic power} Recall that for a
graph $G_N$, $A(G_N)$ is the adjacency matrix  of $G_N$. Now, using graph limit theory we show the following result:

\begin{thm}
\label{dense_testing}
Let $\{G_N\}_{N\geq1}$ be a sequence of simple graphs which converges in cut-metric to $W \in\sW$ such that
$\int_{[0, 1]^2}W(x, y)\,\mathrm{d} x\,\mathrm{d} y>0$. Consider the testing
problem (\ref{testingbeta12}) given a single realization $\sigma\in
S_N$ from (\ref{pmf}) with sufficient statistic $H_N(\tau)=\frac
{1}{N}\tau' A(G_N)\tau$.
\begin{enumerate}[(a)]
\item[(a)] If $0\leq\beta_1<\beta_2<\frac{1}{\llVert W\rrVert }$, then
there does
not exist a sequence of consistent tests for (\ref{testingbeta12}).

\item[(b)] If $\beta_0> \frac{1}{\llVert W\rrVert }$, then the MPLE $\{
\hat\beta_N\}
_{N\geq1}$ is a sequence of $\sqrt{N}$-consistent estimators for
$\beta=\beta_0$.
\end{enumerate}
\end{thm}

The proof of the theorem is given in Section~\ref{pfdense}. It
involves showing that $F_N(\beta_0)=O(1)$ whenever $0 \leq\beta
_0<\frac{1}{\llVert W\rrVert }$, for any converging sequence of graphs,
which together with Theorem~\ref{genfixed} proves (a). To show (b) it
suffices to show that $\lim_{N\rightarrow\infty}\frac
{1}{N}F_N(\beta_0)>0$, for $\beta_0> \frac{1}{\llVert W\rrVert }$
(by Corollary~\ref{sc}). For Ising models on a convergence sequence of
graphs, $\lim_{N\rightarrow\infty}\frac{1}{N}F_N(\beta_0)$ is
given by a
variational problem (\ref{denseopt}) (cf. \cite{graph_limits_II},
Theorem~2.14). Even though explicitly solving this
variational problem
for large values of $\beta$ is extremely difficult, a simple argument
can be used to show that the value of the variational problem is
positive for $\beta> \frac{1}{\llVert W\rrVert }$.

%%%%%%%%%%%%%%%%%%%%%%%%%%%%%%%%%%%%%%%%%%%%%%%%%%%%%%%%%%%%%%%%%%%%%%%%%%%%%%%%%

%\begin{figure*}[h]
%\centering
%\begin{minipage}[c]{1.0\textwidth}
%\centering
%\includegraphics[width=2.8in,height=2.8in]
% {Erdos_Renyi_Power.pdf}
%\end{minipage}
%%\begin{minipage}[c]{0.49\textwidth}
%%\centering
%%\includegraphics[width=3.0in]
%% {KNPowerL.pdf}
%%\end{minipage}
%
%\end{figure*}
%%%%%%%%%%%%%%%%%%%%%%%%%%%%%%%%%%%%%%%%%%%%%%%%%%%%%%%%%%%%%%%%%%%%%%%%%%%%%%%%%%%%%%%%%%%%%%%%%%%%%%%%%%%

By the Neyman--Pearson lemma, the most-powerful (MP) test for (\ref
{testingbeta12}) is based on the sufficient statistic $H_N(\sigma)$.
By Theorem~\ref{dense_testing}, the test based on $H_N(\sigma)$ is
not consistent (see Figure~\ref{randomgraphpower}). However, the
asymptotic power of the MP-test can be derived from the limiting
distribution of $H_N(\sigma)$, for any $\beta< \frac{1}{\llVert
W\rrVert }$.

%
%f2 #&#
%

%%%%%%%%%%%%%%%%%%%%%%%%%%%%%%%%%%%%%%%%%%%%%%%%%%%%%%%%%%%%%%%%%%%%%%%%%%%%%%%
\begin{figure*}[h]
\centering
\begin{minipage}[c]{1.0\textwidth}
\centering
\includegraphics[width=2.8in,height=2.8in]
    {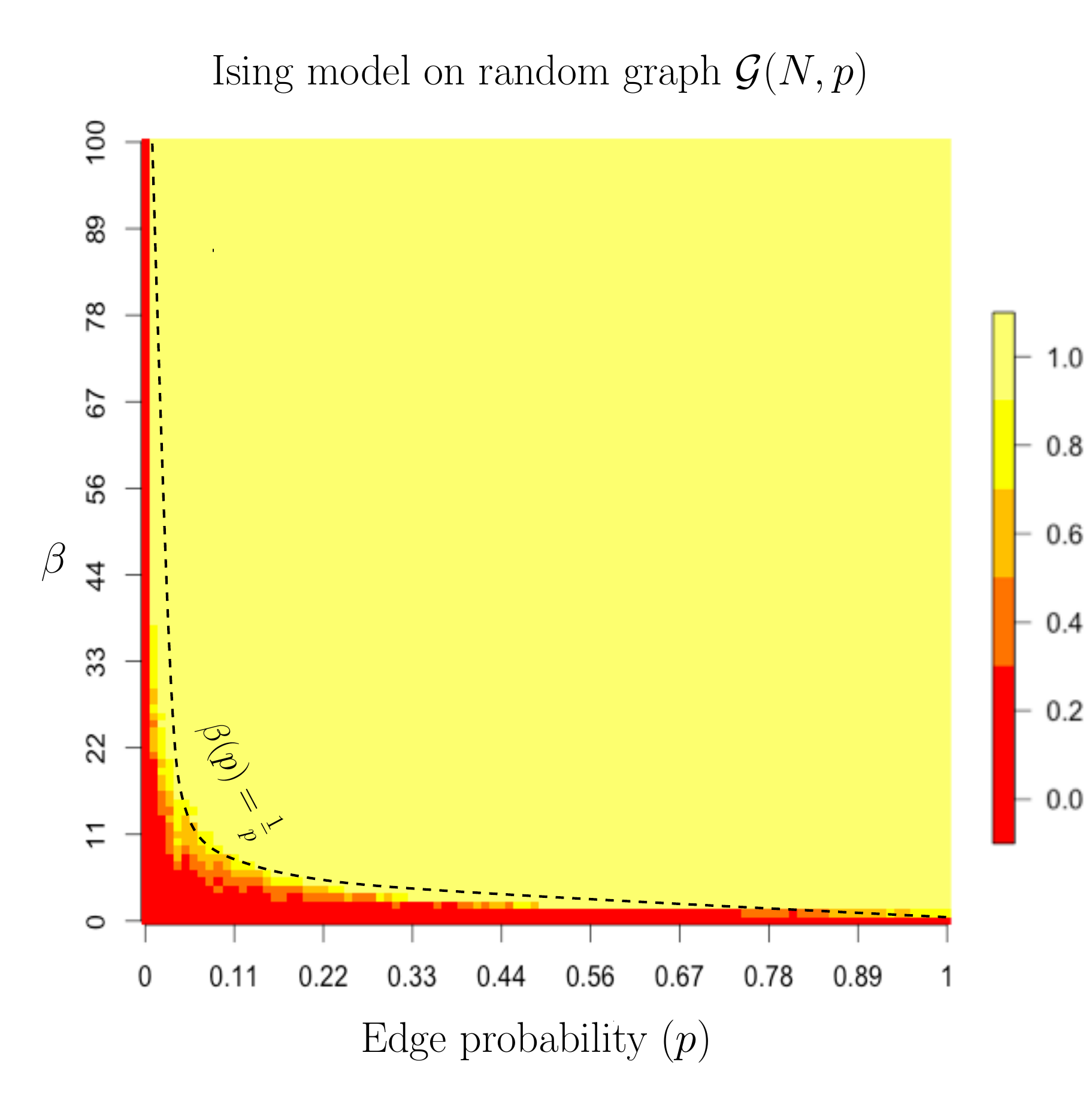}\\
\end{minipage}
%\begin{minipage}[c]{0.49\textwidth}
%\centering
%\includegraphics[width=3.0in]
%    {KNPowerL.pdf}\\
%\end{minipage}
\caption{\footnotesize{The power of the MP-test for the Ising model on an Erd\H os-R\'enyi random graph $\cG(N, p)$ as a function of $p$ and $\beta$, with $N=500$. Every point $(p, \beta)$ in the grid shows the empirical power of the MP-test averaged over 100 repetitions. Note the phase transition curve $\beta(p)=\frac{1}{p}$  above which the MP-test has power 1, as predicted by Theorem~\ref{dense}.}}\vspace{-0.1in}
\label{randomgraphpower}
\end{figure*}
%%%%%%%%%%%%%%%%%%%%%%%%%%%%%%%%%%%%%%%%%%%%%%%%%%%%%%%%%%%%%%%%%%%%%%%%%%%%%%%%%%%%%%%%%%%%%%%%%%%%%%%%%

%th3.4 #&#
%
\begin{thm}\label{dense}
Let $\{G_N\}_{N\geq1}$ be a sequence of simple graphs which converges in cut-metric to $W \in\sW$, with
$\int_{[0, 1]^2}W(x, y)\,\mathrm{d} x\,\mathrm{d} y>0$. If $\sigma
\sim\P
_{\beta}$, then for $\beta<\frac{1}{\llVert W\rrVert }$
%
%e3.7 #&#
%
\begin{equation}
H_N(\sigma)= \frac{1}{N}\sigma'
A(G_N)\sigma\dto\sum_{i=1}^\infty
\lambda_i(W) \biggl(\frac{1}{1-\beta\lambda_i(W)} \xi_i-1 \biggr),
\label{limitdense}
\end{equation}
where $\xi_1,\xi_2,\ldots, $ are i.i.d. $\chi_1^2$ random variables.
\end{thm}

Hereafter, the random variable in the RHS of (\ref{limitdense}) will
be denoted by $Q_{\beta, W}$ and can be used to compute the asymptotic
power for the test based on $H_N(\sigma)$ for the testing problem
(\ref{testingbeta12}), when $0\leq\beta_1<\beta_2<\frac
{1}{\llVert W\rrVert }$. To this end, we need the following definition:

%de3.1 #&#
%
\begin{defn} \label{not}
Let $W\in\sW$ and $\beta< \frac{1}{\llVert W\rrVert }$. Denote by
$F_{\beta,
W}$ the distribution function of the random variable $Q_{\beta, W}$
defined in (\ref{limitdense}). Also, let $q_{1-\alpha, \beta, W}$ be
the $(1-\alpha)$th quantile of $F_{\beta, W}$, that is, $\P_\beta
(Q_{\beta, W} \geq q_{1-\alpha, \beta, W})=\alpha$.
\end{defn}

The following corollary is an immediate consequence of the
Neyman--Pearson lemma and Theorem~\ref{dense}.

%co3.5 #&#
%
\begin{cor}Fix $\alpha\in(0, 1)$ and $0\leq\beta_1<\beta_2<\frac
{1}{\llVert W\rrVert }$. The most powerful level $\alpha$ test for
(\ref
{testingbeta12}) rejects $H_0$ when $H_N(\sigma)>q_{1-\alpha, \beta
_1, W}$, and has limiting power
%
%e3.8 #&#
%
\begin{equation}
\lim_{N\rightarrow\infty}\P_{\beta_2}\bigl(H_N(
\sigma)>q_{1-\alpha,
\beta_1, W}\bigr)=1-Q_{\beta_2, W}(q_{1-\alpha, \beta_1, W}).
\label{pb1}
\end{equation}
\end{cor}

%%%%%%%%%%%%%%%%%%%%%%%%%%%%%%%%%%%%%%%%%%%%%%%%%%%%%%%%%%%%%%%%%%%%%%%%%%%%%%%%%%
%\begin{figure*}[h]
%\centering
%\begin{minipage}[c]{1.0\textwidth}
%\centering
%\includegraphics[width=2.8in]
% {KNPowerL.pdf}
%\end{minipage}
%%\begin{minipage}[c]{0.49\textwidth}
%%\centering
%%\includegraphics[width=3.0in]
%% {KNPowerL.pdf}
%%\end{minipage}
%
%\end{figure*}
%%%%%%%%%%%%%%%%%%%%%%%%%%%%%%%%%%%%%%%%%%%%%%%%%%%%%%%%%%%%%%%%%%%%%%%%%%%%%%%%%%%%%%%%%%%%%%%%%%%%%%%%%%%

In most of the relevant examples, the limiting graphon $W$ has finitely
many non-zero eigenvalues, and the expression on the RHS of (\ref
{pb1}) can be computed easily in terms of the quantiles of the
chi-squared distribution.

%ex2 #&#
%
\begin{example}Suppose $G_N\sim\cG(N, p)$ be a Erd\H os--R\'enyi
random graph with $0< p \leq1$. Then $G_N$ converges to the constant
function $W_p:\equiv p$ on $[0, 1]^2$, which has only one non-zero
eigenvalue $\lambda_1(W_p)=p$. Therefore, consistent testing is
impossible for $0\leq\beta<\frac{1}{p}$ (see Figure~\ref
{randomgraphpower}). Moreover, for $\beta< 1/p$, (\ref{pb1})
simplifies to
\[
H_N(\sigma) \dto Q_{\beta, W_p}=p \biggl(\frac{1}{1-\beta p} \chi
^2_{1}-1 \biggr).
\]
If $q_{1-\alpha}$ denotes the $(1-\alpha)$th quantile of the $\chi
^2_1$ distribution, then by (\ref{pb1}), the limiting power of the
test with rejection region $\{H_N(\sigma)>c_\alpha:=p(q_{1-\alpha
}-1)\}$ for the testing problem $\beta=0$ versus $\beta=\beta_0<
1/p$ is
%
%e3.9 #&#
%
\begin{equation}
\lim_{N\rightarrow\infty}\P_{\beta_0}\bigl(H_N(
\sigma)>c_{\alpha}\bigr)=\P\bigl(\chi^2_1>(1-
\beta_0 p)q_{1-\alpha}\bigr). \label{kngraphpower}
\end{equation}
The limiting power of the MP-test for the Curie--Weiss model (which
corresponds to taking $p=1$ in \eqref{kngraphpower}) is shown in
Figure~\ref{limitingpower}. Note that it has a phase transition at
$\beta=1$, as stated in Theorem~\ref{dense_testing}.
\end{example}

%f3 #&#
%

%%%%%%%%%%%%%%%%%%%%%%%%%%%%%%%%%%%%%%%%%%%%%%%%%%%%%%%%%%%%%%%%%%%%%%%%%%%%%%%
\begin{figure*}[h]
\centering
\begin{minipage}[c]{1.0\textwidth}
\centering
\includegraphics[width=2.8in]
    {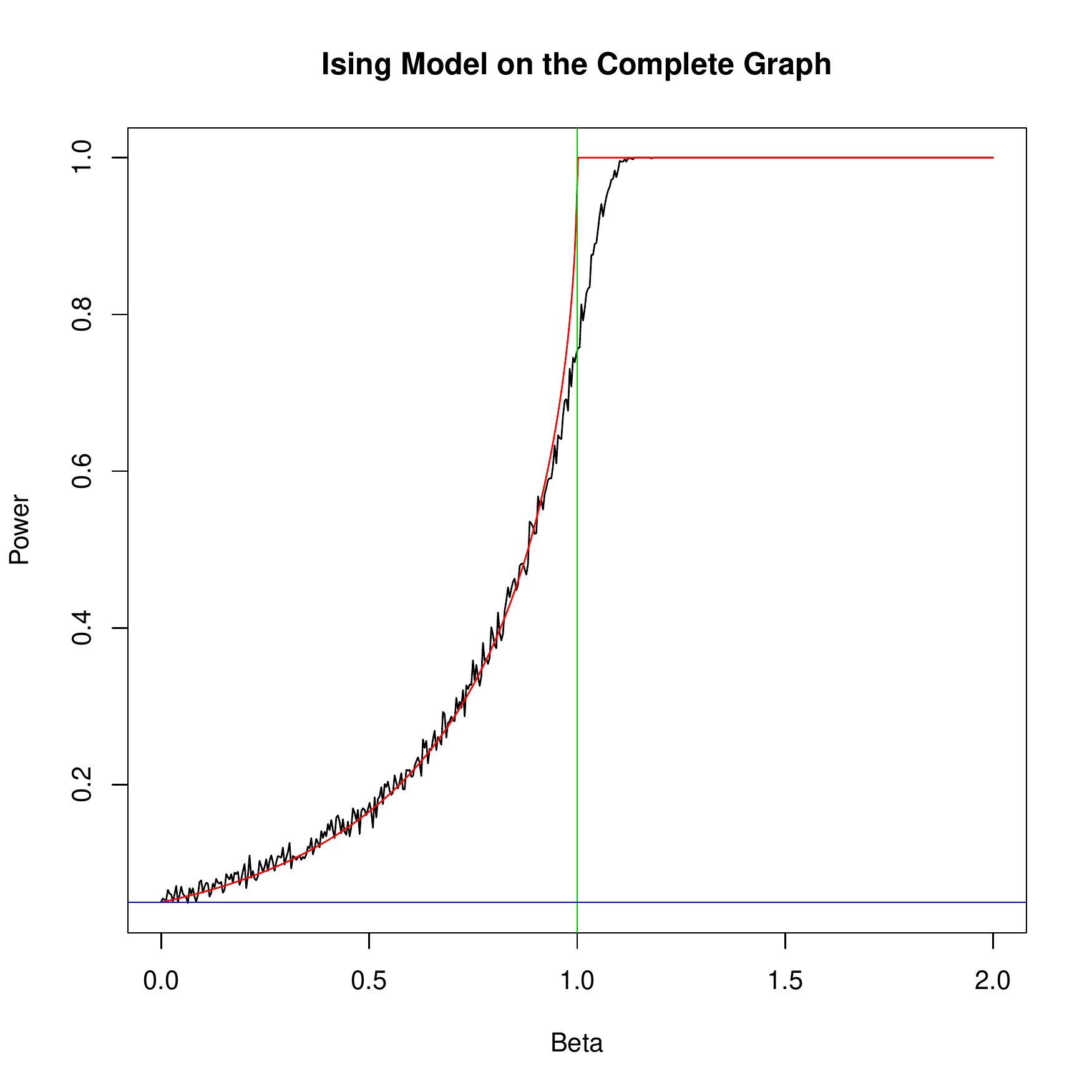}\\
\end{minipage}
%\begin{minipage}[c]{0.49\textwidth}
%\centering
%\includegraphics[width=3.0in]
%    {KNPowerL.pdf}\\
%\end{minipage}
\caption{\footnotesize{The power of the MP-test in the Curie Weiss Model as a function of $\beta$; the black curve is the empirical power for the Curie-Weiss model with $N=500$ and 1000 repetitions at each point along a sequence of values (of length 500) of $\beta \in [0, 2]$. The red curve is the limiting power function (corresponds to taking $p=1$ in (\ref{kngraphpower})) as a function of $\beta\in [0, 2]$. The blue line corresponds to the level $\alpha=0.05$ of the test.}}
\label{limitingpower}\vspace{-0.1in}
\end{figure*}
%%%%%%%%%%%%%%%%%%%%%%%%%%%%%%%%%%%%%%%%%%%%%%%%%%%%%%%%%%%%%%%%%%%%%%%%%%%%%%%%%%%%%%%%%%%%%%%%%%%%%%%%%

%re3.2 #&#
%
\begin{remark} Note that throughout the paper, the term \textit{phase
transition} has been is used to imply a change in the rate of
consistency of the pseudo-likelihood estimate  $\hat{\beta}_N$. Interestingly, in
all our examples (Corollaries \ref{dregular}, \ref
{erdosrenyirandom} and Theorem~\ref{dense_testing}) the change in the
rate of consistency happens exactly at the point of thermodynamic phase
transition, that is, prior to this phase transition point the
log-partition function is $o(N)$, whereas after the phase-transition
point the log-partition function scales linearly with $N$. In fact, in
the setting of Corollary~\ref{dregular}, the limiting log-partition
function is continuous but not differentiable at the phase transition
point $\beta=1$ (see \cite{BM}, Theorem~2.2(b)). Similar statements
about the non-differentiability of the limiting log-partition function
should also hold for the other two examples, but since they are not directly used in our calculations, this direction has
not been pursued.
\end{remark}

%s4 #&#
\section{Analysis of the facebook dataset}
\label{sec:datasets}

Ising models have been widely used to understand correlations among
neighboring vertices in network data with binary node attributes. Here,
we use it to study the effect of gender in Facebook friendship-networks
using data from the Stanford Large Network Dataset (SNAP) collection,
available freely at \url
{http://snap.stanford.edu/data/egonets-Facebook.html}. The nodes are
groups of users from Facebook and there is an edge between two users if
they are friends. The dataset also include several anonymized node
features, such as hometown, gender, birthday, school, and university.
We consider two networks (referred to as FB1 and FB2) with gender as
the binary node feature, encoding, without loss of generality, male by
1 and female by $-1$. The nodes labelled 1 are colored blue and those
labelled $-1$ are colored red. The FB1 network has 221 nodes and 3176
edges. Among the 221 nodes, 170 are labelled 1 and 51 are labelled $-1$.
The FB2 network has 333 nodes and 2519 edges, with 213 nodes labelled 1
and 120 labelled~$-1$.\looseness=1

%%%%%%%%%%%%%%%%%%%%%%%%%%%%%%%%%%%%%%%%%%%%%%%%%%%%%%%%%%%%%%%%%%%%%%%%%%%%%%%%%

%\begin{figure}
%\begin{tabular}{@{}c@{}}
%\begin{tabular*}{\textwidth}{@{\extracolsep{\fill}}lll@{}}
%\hline
% & FB1 & FB2 \\
%\hline
%(Vertices, Edges)& (221, 3176) & (333, 2519) \\
%Average Degree & 28.74 & 15.13\\
%%\hline
%%$\lambda_{\max}^{-1}$ & 0.5771 & 0.4079\\
%MPLE & 1.0518 & 0.8530\\
%$p$-value & 0.0045 & 0.0001\\
%\hline
%\end{tabular*} \\[6pt]
% \figlink{886f04}
%\end{tabular}
%\caption{Facebook friendship-network: The table gives the MPLE of $
%\beta$ for
%Ising models on two Facebook friendship-networks and the corresponding
%$p$-values for
%testing independence. The plot shows the empirical (resampled)
%error-bars for the MPLE in the two networks.}
%\label{fig:fbestimates}
%\end{figure}

%\begin{figure}
%\begin{tabular}{@{}c@{ }c@{}}
%\begin{minipage}{120pt}
%\begin{figure}
%\figlink{886f04}
%\end{figure}
%\end{minipage}
%&
%\begin{minipage}{210pt}
%
%\vspace*{9.5pt}\end{minipage}
%\end{tabular}
%
%\end{figure}

%
%\begin{figure*}[h]
%
%\centering
%\begin{minipage}[l]{0.495\textwidth}
%\scriptsize
%
%\end{minipage}
%\begin{minipage}[r]{0.495\textwidth}
%\includegraphics[width=2.2in]
% {FB_Estimates.pdf}
%\end{minipage}
%
%\end{figure*}
%%%%%%%%%%%%%%%%%%%%%%%%%%%%%%%%%%%%%%%%%%%%%%%%%%%%%%%%%%%%%%%%%%%%%%%%%%%%%%%%%%%%%%%%%%%%%%%%%%%%%%%%%%%%

In order to understand how gender correlates with friendship, we fit
Ising models on the two networks. The MPLE for $\beta$ corresponding
to the two networks are given in the table in Figure~\ref
{fig:fbestimates}. This can be used to test the null hypothesis that
gender does not correlate with friendship. The $p$-values show that the
null hypothesis is rejected at the 5\% level in both cases, suggesting,
as expected, significant correlation in the friendship-network based on
gender. The MPLE in FB1 is larger, which suggests a stronger
gender-based correlation in FB1, which might be due to the larger
male-to-female ratio in FB1 than in FB2.

%%%%%%%%%%%%%%%%%%%%%%%%%%%%%%%%%%%%%%%%%%%%%%%%%%%%%%%%%%%%%%%%%%%%%%%%%%%%%%%
\begin{figure*}[h]
\vspace{-0.25in}
\centering
\begin{minipage}[l]{0.495\textwidth}
\scriptsize
\begin{tabular}{c|c|c}
\hline
 & FB1 &  FB2 \\
\hline
(Vertices, Edges)& (221, 3176) &  (333, 2519) \\
\hline
Average Degree & 28.74 & 15.13\\
%\hline
%$\lambda_{\max}^{-1}$ & 0.5771 & 0.4079\\
\hline
MPLE &  1.0518 & 0.8530\\
\hline
$p$-value & 0.0045  & 0.0001\\
\hline
\end{tabular}
\end{minipage}
\begin{minipage}[r]{0.495\textwidth}
\includegraphics[width=2.2in]
    {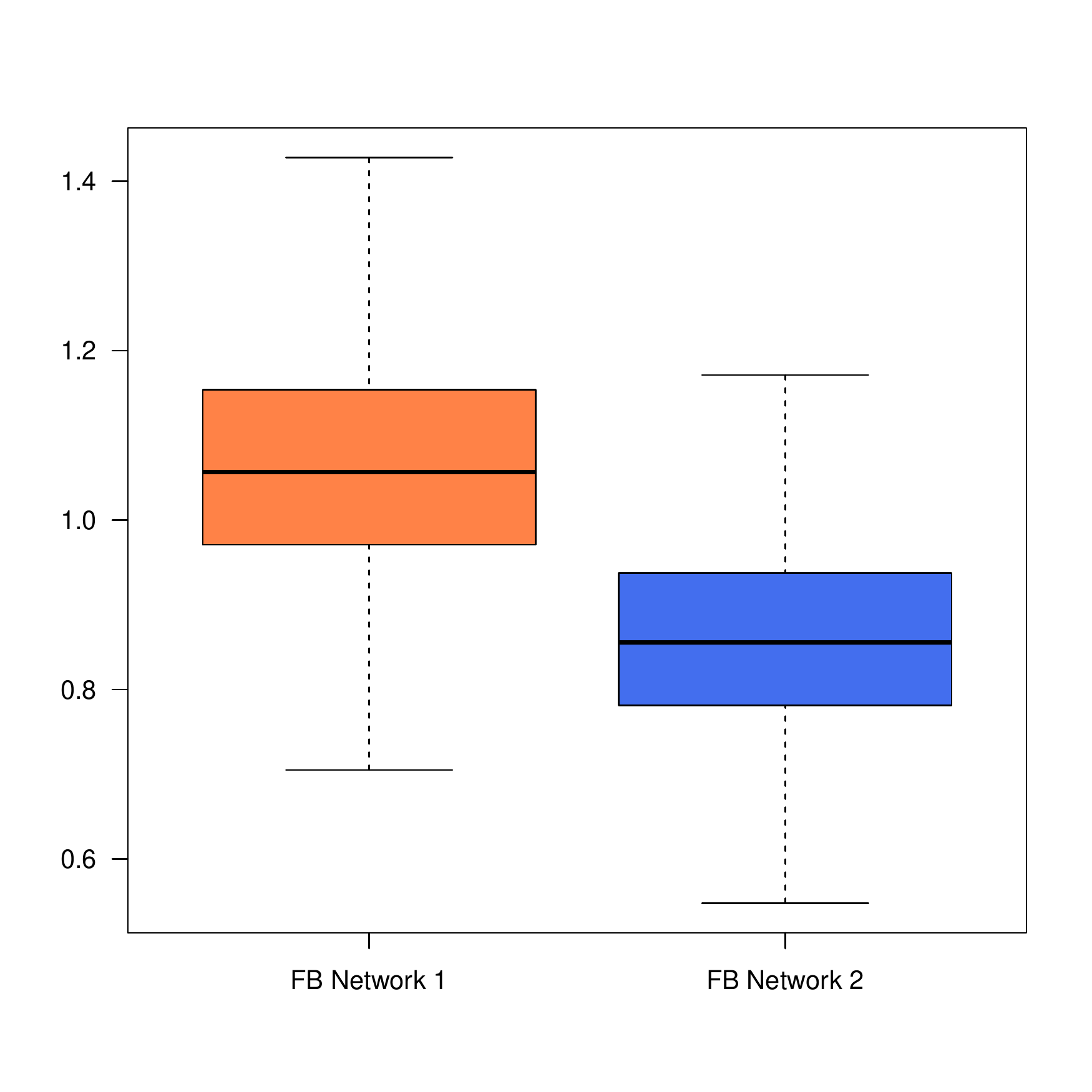}\\ 
\end{minipage}\vspace{-0.2in}
\caption{\footnotesize{Facebook friendship-network: The table gives the MPLE of $\beta$ for Ising models on two Facebook friendship-networks and the corresponding $p$-values for testing independence. The plot shows the empirical (resampled) error-bars for the MPLE in the two networks.}}
\label{fig:fbestimates}\vspace{-0.15in}
\end{figure*}
%%%%%%%%%%%%%%%%%%%%%%%%%%%%%%%%%%%%%%%%%%%%%%%%%%%%%%%%%%%%%%%%%%%%%%%%%%%%%%%%%%%%%%%%%%%%%%%%%%%%%%%%%

Figure~\ref{fig:fbestimates} also shows the error bars for the MPLE
calculated using parametric bootstrap: $10^5$ realizations of the Ising
model were resampled using the original MPLE, which then gives an
estimate of the standard error of the MPLE. Note that the error bar for
FB1 is slightly longer than that for FB2. This might be because the FB1
network, with average degree 28.74, is significantly dense than FB2,
which has average degree 15.13. %This is also supported by Theorem~
\section{Proof of consistency of the MPLE}
\label{pfmain}

This section contains the proof of Theorem~\ref{GEN}. The technical
lemmas required for the proof are listed in Section~\ref
{pftechnicallemmas} and proved later in Appendix \ref
{pflem1}. Using this, we complete the proof of the theorem
in Section~\ref{pfgen}. Corollary~\ref{sc} is proved in Section~\ref
{pfcorchatterjee}.

%\begin{remark}
%It suffices to prove theorem \ref{GEN} for $\beta_0>0$. This is
%because if $\beta_0<0$ then defining $\widetilde{\beta}=-\beta$ we
%have $\widetilde{\beta}_0=-\beta_0>0$. Also setting $
%\widetilde{J}_N=-J_N$ it follows that $\widetilde{J}_N$ has same
%operator norm and Frobenius norm as $J_N$. With this in mind the rest
%of this section assumes $\beta_0>0$. As a consequence of this
%assumption we can replace \eqref{maincond} by the condition
%\begin{align}\label{maincond_positive}
%&0<\lim_{\delta\rightarrow0}\liminf_{N\rightarrow\infty}
%\frac{1}{a_N}F_N(\beta_0-\delta)\le\lim_{\delta\rightarrow0}\limsup_{N
%\rightarrow\infty}\frac{1}{a_N}F_N(\beta_0+\delta)<\infty.
%\end{align}

%\end{remark}

%s5.1 #&#
\subsection{Technical lemmas}
\label{pftechnicallemmas}

%We begin by proving the following Lemma~\ref{lem1}, part (b) of which
%generalizes the conclusion of \cite{Chatterjee}, Lemma~2.3, 2.4, 2.5,
%to give a simpler proof of the fact
%that a lower bound to the order of the sufficient statistic $H_N(
%\sigma)$ is given by the order of the log-partition function, i.e. the
%sequence of random variables $\frac{1}{a_N}H_N(\sigma)$ do not tend to
%$0$ in distribution. Part (c) of the Lemma shows that the lower bound
%is tight, in the sense that the sequence $\frac{1}{a_N}H_N(\sigma)$ is
%$O_P(1)$. Parts (b) and (c) together thus determine the correct order
%of fluctuation of $H_N(\sigma)$ for all Ising models satisfying
%\eqref{maincond}

The proof of Theorem~\ref{GEN} requires a few technical lemmas. We
begin by showing that in Ising models satisfying (\ref{maincond}), the
asymptotic order of the sufficient statistic $H_N(\sigma)$ is the same
as the order of the log-partition function, that is, (a) the sequence
$\frac{1}{a_N}H_N(\sigma)$ does not tend to $0$ in distribution, and
(b) $\frac{1}{a_N}H_N(\sigma)$ is $O_P(1)$. In fact, (b) is not
required in the rest of the proof, however we include it because,
together with (a), it gives the correct order of $H_N(\sigma)$, which
appears to be new and might be of independent interest. The proof of
the lemma is given in Appendix \ref{pflem1}.

%le5.1 #&#
%
\begin{lem}\label{lem1}
Under assumption \eqref{maincond}, the following hold:
\begin{enumerate}[(a)]
\item[(a)]$\lim_{\varepsilon\rightarrow0}\limsup_{N\rightarrow\infty
}\P_{\beta_0}(H_N(\sigma)<\varepsilon a_N)=0$,

\item[(b)]$\lim_{K\rightarrow\infty}\limsup_{N\rightarrow\infty}\P
_{\beta_0}(H_N(\sigma)>K a_N)=0$.
\end{enumerate}
\end{lem}

The next lemma is similar to the lemma in \cite{Chatterjee},
Lemma~1.2, where it was shown that the second moment of the
function $L_\sigma(\beta_0)$ is $O(1/N)$ whenever the log-partition
function scales like $N$. Here, by a finer analysis using part (a) of
Lemma~\ref{lem1} we show that the $\E_{\beta_0} (L_\sigma(\beta
_0)^2)=O(a_N/N^2)$, if the log-partition function has order $a_N$. The
proof of the lemma is given in Appendix \ref{pflem2moment}.

%le5.2 #&#
%
\begin{lem}
\label{2moment}
Let $L_\sigma$ be as defined in \eqref{s}. Then under the assumptions
in Theorem~\ref{GEN}, for $N$ large enough,
\[
\limsup_{N\rightarrow\infty}\frac{N^2}{a_N}\E_{\beta_0}
\bigl(L_\sigma(\beta_0)^2\bigr) <\infty.
\]
%
%where $C$ is a constant which depends on $\beta_0$ and $.%\footnote{
%\blue{Note that the constant $C(\beta_0)$ may also depend on the
%matrix $J_N$. However, the dependence on $J_N$ is supressed, because
%the result holds conditional on $J_N$ fixed.}}
\end{lem}

%le5.3 #&#
%
\begin{lem}\label{derivative} Under the assumptions in Theorem~\ref
{GEN},
\[
\lim_{\varepsilon\rightarrow0}\lim_{K\rightarrow\infty}\lim
_{N\rightarrow\infty}\P_{\beta_0} \Biggl(\sum
_{i=1}^N m_i(\sigma)^2
\pmb1\bigl\{\bigl|m_i(\sigma)\bigr|\le K\bigr\}\le\varepsilon a_N
\Biggr)=0.
\]
\end{lem}

The above lemma replaces the application of Paley--Zygmund inequality of
\cite{Chatterjee}, Lemma~2.2, and will be used to complete the proof
of Theorem~\ref{GEN}.

%s5.2 #&#
\subsection{Completing the proof of Theorem \protect\ref{GEN}}
\label{pfgen}

By Chebyshev's inequality and Lemma~\ref{2moment} there exists
$C<\infty$ such that
%
%e5.1 #&#
%
\begin{eqnarray}
\P_{\beta_0}\bigl(\bigl|L_\sigma(\beta_0)\bigr|>K_1
\sqrt{a_N}/N\bigr)\le\frac
{N^2}{a_N}\E_{\beta_0}
L_\sigma(\beta_0)^2 \le\frac{C}{K_1^2}.
\label{chebyshev}
\end{eqnarray}
Now, fix $\delta>0$. Therefore, it is possible to choose
$K_1=K_1(\delta)$ such that the RHS above is less than $\delta$.

Also, by Lemma~\ref{derivative} there exists $\varepsilon
:=\varepsilon(\delta)>0$ and $K_2=K_2(\varepsilon,\delta)<\infty$
such that
%
%e5.2 #&#
%
\begin{eqnarray}
\P_{\beta_0} \Biggl(\sum_{i=1}^N
m_i(\sigma)^2\pmb1\bigl\{\bigl|m_i(\sigma)\bigr|\le
K_2\bigr\}\ge\varepsilon a_N \Biggr)\ge1-\delta,
\label{epsilon}
\end{eqnarray}
for $N$ large enough. Thus, taking $N$ large enough and setting
\[
T_N(\beta_0):= \Biggl\{\sigma\in S_N:
\bigl|L_\sigma(\beta_0)\bigr| \le K_1\frac{\sqrt{a_N}}{N},
\sum_{i=1}^N m_i(
\sigma)^2 \pmb1\bigl\{ \bigl|m_i(\sigma)\bigr|\le K_2
\bigr\}\ge\varepsilon a_N \Biggr\},
\]
we have $\P_{\beta_0}(T_N)\ge1-\delta$. For $\sigma\in T_N$,
%
%e5.3 #&#
%
\begin{eqnarray}\label{diffpl}
\bigl|L_\sigma'(\beta_0)\bigr|&=&\biggl\vert
\frac{\partial}{\partial\beta
}L_\sigma(\beta)\biggr\vert_{\beta=\beta_0}=
\frac{1}{N}\sum_{i=1}^N
m_i(\sigma)^2\operatorname{sech}^2\bigl(
\beta_0 m_i(\sigma)\bigr)
\nonumber
\\
&\ge& \frac{1}{N}\operatorname{sech}^2(\beta_0
K_2)\sum_{i=1}^N
m_i(\sigma)^2 \pmb1\bigl\{\bigl|m_i(\sigma)\bigr|
\le K_2\bigr\}
\\
&\ge& \varepsilon\frac{a_N}{N}\operatorname{sech}^2(
\beta_0 K_2). \nonumber
\end{eqnarray}
Therefore,
\begin{eqnarray*}
K_1\frac{\sqrt{a_N}}{N}&\ge&\bigl|L_{\sigma}(\beta_0)\bigr|=\bigl|L_{\sigma
}(
\beta_0)-L_{\sigma}\bigl(\hat\beta_N(\sigma)
\bigr)\bigr|
\nonumber
\\
&\ge&
\int_{\beta_0\wedge\hat\beta_N(\sigma)}^{\beta_0\vee\hat
\beta_N(\sigma)} L'_\sigma(
\beta)\,\mathrm{d}\beta
\\
&\ge& \frac{\varepsilon a_N}{K_2 N}\bigl|\tanh\bigl(K_2\hat\beta
_N(\sigma
)\bigr)-\tanh(K_2\beta_0)\bigr|.\nonumber
\end{eqnarray*}
Let $R=R(\delta):=\frac{K_2}{K_1\varepsilon}$. This implies that
%
%e5.4 #&#
%
\begin{equation}
\P_{\beta_0} \bigl(\sqrt{a_N}\bigl|\tanh(K_2\hat
\beta_N)-\tanh(K_2\beta_0)\bigr|\geq R \bigr)\leq
\delta, \label{consistent}
\end{equation}
and Theorem~\ref{GEN} follows.

%s5.3 #&#
\subsection{Proof of Corollary \protect\ref{sc}}
\label{pfcorchatterjee}

Note that for $\tau\in S_N$ and any $K>0$,
\[
\sum_{i=1}^N \bigl|m_i(\tau)\bigr|\pmb
1\bigl\{\bigl|m_i(\tau)\bigr|>K\bigr\}\leq\frac
{1}{K}\sum
_{i=1}^N m_i(\tau)^2 =
\frac{1}{K}\tau'J_N^2\tau\leq
\frac{N \llVert J_N\rrVert ^2}{K}.
\]
Therefore, condition (a) in Theorem~\ref{GEN} holds with $a_N=N$.

Moreover,
\[
\sum_{ i, j=1}^N J_N^2(i,
j)=\sum_{i=1}^N \llVert J_N
\vec e_i\rrVert_2^2\le\llVert
J_N\rrVert^2 \sum_{i=1}^N
\llVert\vec e_i\rrVert^2_2= N\llVert
J_N\rrVert^2,
\]
that is condition (b) in Theorem~\ref{GEN} holds with $a_N=N$.

Finally, to check \eqref{maincond} note that $F_N'(\beta)=\frac
{1}{2}\E_\beta\sigma'J_N\sigma\le\frac{M}{2}N$, where
$M:=\llVert J_N\rrVert <\infty$. Therefore,
\[
\lim_{\delta\rightarrow0}\liminf_{N\rightarrow\infty}
\frac
{1}{N}F_N(\beta_0-\delta)\ge\lim
_{\delta\rightarrow0}\liminf_{N\rightarrow\infty} \biggl(
\frac{1}{N}F_N(\beta_0)-\frac{M}{2} \delta
\biggr)>0,
\]
by condition (\ref{ptN}). Also, $\lim_{\delta\rightarrow0}\limsup
_{N\rightarrow\infty}\frac{1}{N}F_N(\beta_0+\delta)\le M\lim
_{\delta\rightarrow0} (\beta_0+\delta)<\infty$.
This verifies (\ref{maincond}) and by an application of Theorem~\ref
{GEN} the result follows.

%%%%%%%%%%%%%%%%

%s6 #&#
\section{Proof of Theorem \protect\ref{genfixed}}
\label{pfgenfixed}

In this section, we give the proof of Theorem~\ref{genfixed}, which
shows that consistent testing and estimation is impossible whenever the
partition function is $O(1)$. This is a consequence of a general result
(see Proposition~\ref{TESTING} below) which shows that distinguishing
two probability measures $\P_N$ versus $\Q_N$ is impossible whenever
the KL divergence between the two measures $\P_N$ and $\Q_N$ remains
asymptotically bounded.

%s6.1 #&#
\subsection{Non-existence of consistent tests}

For every $N\geq1$, let $(\sX_N, \cF_N)$ be a measure space and $\P
_N$ and $\Q_N$ two distributions on this measure space. Let $\mu_N$
be a dominating measure for both $\P_N$ and $\Q_N$, and $p_N$ and
$q_N$ denote the respective densities with respect to this measure.
Also, denote the Kullback--Leibler (KL) divergence between $\Q_N$ and
$\P_N$ by
%
%e6.1 #&#
%
\begin{eqnarray}
\label{kldivg} D(\Q_N\|\P_N)&:=&\E_{\Q_N}L_N(
\pmb X):=\E_{\Q_N}\log\frac
{q_N(\pmb X)}{p_N(\pmb X)}
\nonumber
\\[-8pt]
\\[-8pt]
\nonumber
&=&
\int_{\sX_N}q_N({\pmb x})\log\frac{q_N({\pmb x})}{p_N({\pmb
x})}
\,\mathrm{d}\mu_N.
\end{eqnarray}

Consider the problem of testing $\P_N$ versus $\Q_N$. A sequence of
tests $\phi_N$ is \textit{consistent} for this testing problem if there
exists a sequence of test functions $\{\phi_N\}_{N\geq1}$ such that
$\lim_{N\rightarrow\infty}\E_{\P_N}\phi_N=0$, and $\lim
_{N\rightarrow\infty}\E_{\Q_N}\phi_N=1$. %The testing problem is
%said to be \textit{inconsistent} if there exists no consistent sequence
%of test functions.

%pr6.1 #&#
%
\begin{ppn}\label{TESTING}
Consider the problem of testing $\P_N$ versus $\Q_N$. If
%
%e6.2 #&#
%
\begin{equation}
\limsup_{N\rightarrow\infty} D(\Q_N\|\P_N)<\infty,
\label{divlemma}
\end{equation}
then there does not exist a consistent sequence of tests for this
testing problem.
\end{ppn}

The proof of the proposition is given in Appendix \ref{pftesting}. In
the following, we use it to prove Theorem~\ref{genfixed}.

%s6.2 #&#
\subsection{Completing the proof of Theorem \protect\ref{genfixed}}
Given Proposition~\ref{TESTING}, it remains to verify that
%
%e6.3 #&#
%
\begin{equation}
\label{beta12KL} D(\P_{\beta_1}\|\P_{\beta_2})=F_N(
\beta_2)-F_N(\beta_1)-(\beta_2-
\beta_1)F_N'(\beta_1)< \infty,
\end{equation}
for $0\leq\beta_1< \beta_2\leq\beta_0$ (where $\beta_0$ satisfies
\eqref{maincondfixed}).

By hypothesis \eqref{maincondfixed} there exists $M<\infty$ such that
$F_N(\beta_1)< M$ and $F_N(\beta_2)< M$, for $N$ large enough.
Moreover, by the monotonicity of $F_N'(\cdot)$,
\[
(\beta_2-\beta_1)F_N'(
\beta_1)\leq
\int_{\beta_1}^{\beta_2} F_N'(\theta)
\,\mathrm{d}\theta=F_N(\beta_2)-F_N(
\beta_1)< M,
\]
proving \eqref{beta12KL}.

%s7 #&#
\section{Applications: Proofs of Corollary \protect\ref
{mplepositive}, \protect\ref{dregular} and \protect\ref{erdosrenyirandom}}
\label{pfapplication}

In this section, we prove Corollary~\ref{mplepositive} which will then
be used to derive rates of consistency of the MPLE for Ising models on
different graph ensembles, using Theorems \ref{GEN} and \ref
{genfixed}. To apply these results, we need to determine the correct
order of $F_N(\beta_0)$ in a neighborhood of a point $\beta_0>0$.
However, the exact asymptotics $F_N(\beta_0)$ is known only for
specific choices of the matrix $J_N$ and for specific values of $\beta_0$.

Nevertheless, the correct order of $F_N(\beta_0)$ can be easily
obtained in various examples, using, for instance, the following very
useful lemma, which is of independent interest and may find other applications.

%le7.1 #&#
%
\begin{lem}\label{bound} Consider the family of probability
distributions on $S_N$ given by (\ref{pmf}). Assume that the elements
of the matrix $J_N$ are non-negative, and $\lambda_1(J_N)\leq\lambda
_2(J_N)\leq\cdots\leq\lambda_N(J_N)$ are the eigenvalues of the
matrix $J_N$.
\begin{enumerate}[(a)]
\item[(a)] For $0<\beta< \frac{1}{ \llVert J_N\rrVert }$,
%
%e7.1 #&#
%
\begin{align}
\label{upper_bound} F_N(\beta)\le-\frac{1}{2}\sum
_{i=1}^n\log\bigl(1-\beta\lambda
_i(J_N) \bigr).
\end{align}

\item[(b)] For any $\beta> 0$,
%
%e7.2 #&#
%
\begin{align}
\label{lower_bound}F_N(\beta)\ge\sum_{1\le i< j\le
N}
\log\cosh\bigl(\beta J_N(i, j) \bigr).
\end{align}
\end{enumerate}
\end{lem}

\begin{proof}
Let $W:=(W_1, W_2, \ldots,W_N)'$ be a vector of i.i.d. $N(0,1)$ random
variables. Note that for any $s\ge1$ and non-negative integers $b_1,
b_2, \ldots, b_s$
\[
\E_0 \sigma_1^{b_1}\sigma_2^{b_2}
\cdots\sigma_s^{b_s}\le\E W_1^{b_1}W_2^{b_2}
\cdots W_s^{b_s}.
\]
Since the matrix $J_N$ has non-negative entries, by expanding the
exponential function in power series every term can be bounded using
the above inequality. This implies that
%
%e7.3 #&#
%
\begin{align}
\label{eq:upper_bound} e^{F_N(\beta)}= \E_{0}e^{\frac{1}{2}\beta
\sigma'J_N \sigma}\le\E
e^{\frac{1}{2}\beta W'J_NW}.
\end{align}
The RHS of \eqref{eq:upper_bound} can be computed exactly as follows:
Let $J_N=\sum_{i=1}^N \lambda_i(J_N)p_i p_i'$, be the spectral
decomposition of $J_N$, where $p_1, p_2, \ldots, p_N$ are the
normalized eigenvectors of $J_N$. Then setting $p_i' W=Z_i$ for $1\leq
i \leq N$, we get
%
%e7.4 #&#
%
\begin{align}
\E e^{\frac{1}{2}\beta W 'J_N W}=\E e^{\frac{1}{2}\beta\sum
_{i=1}^N\lambda_i(J_N)(p_i' W)^2}=\E e^{\frac{1}{2}\beta\sum
_{i=1}^N\lambda_i(J_N)Z_i^2}. \label{expnorm}
\end{align}
Note that $Z:=(Z_1, Z_2, \ldots,F_N)$ is a vector of i.i.d. $N(0,1)$
random variables. Therefore, \eqref{eq:upper_bound} and (\ref
{expnorm}) implies
\[
\E e^{\frac{1}{2}\beta W 'J_N W} \le\prod_{i=1}^N
\bigl(1- \beta\lambda_i(J_N) \bigr)^{-\frac{1}{2}},
\]
using the MGF of the chi-squared distribution (since $\beta\lambda
_i(J_N)<1$, for all $1\leq i \leq N$). The inequality \eqref
{upper_bound} follows by taking $\log$ on both sides.

To prove (b), let $\{Y_{ij},1\le i<j\le N\}$ be i.i.d. with $\P
(Y_{ij}=\pm1)=\frac{1}{2}$. Then for any collection of non-negative
integers $((b_{ij}))_{1\leq i< j \leq N}$,
\[
\E\prod_{1\le i<j\le N} Y_{ij}^{b_{ij}}\le
\E_0\prod_{1\le i<j\le
N}(\sigma_i
\sigma_j)^{b_{ij}}.
\]
Indeed, this follows on noting that both the LHS and RHS are $\{0,1\}
$-valued, and the LHS is $1$ if and only if $b_{ij}$ is even for all
$(i,j)$, which is when the RHS is $1$ as well. This implies,
\begin{align*}
e^{F_N(\beta)}={}&\E_{0}\prod_{1\le i< j\le N}e^{\beta J_N(i,j)\sigma
_i\sigma_j }
\ge \E\prod_{1\le i< j\le N}e^{\beta J_N(i, j)Y_{ij} }
\nonumber
\\
={}&\prod_{1\le i< j\le N}\cosh\bigl(\beta J_N(i,
j) \bigr).
\end{align*}
The inequality \eqref{lower_bound} follows on taking $\log$ on both sides.
\end{proof}

%re7.1 #&#
%
\begin{remark}
Note that the upper bound \eqref{upper_bound} is obtained by replacing
the spin configuration $\sigma=(\sigma_1, \sigma_2, \ldots, \sigma
_N)$ with a vector of i.i.d. $N(0, 1)$ random variables. To get the
lower bound, the collection $\{\sigma_i\sigma_j\}_{1\leq i< j \leq
N}$ is replaced by i.i.d. Rademacher random variables. Surprisingly,
the bounds obtained by these simple comparison techniques often give
the correct asymptotic order of $F_N(\beta)$ in the high temperature
regime $\beta<\frac{1}{\llVert J_N\rrVert }$. To get the order of
$F_N(\beta)$
beyond the phase transition, the standard mean-field approximation can
be used (see Section~\ref{bound} for details).
\end{remark}

%s7.1 #&#
\subsection{Proof of Corollary \protect\ref{mplepositive}}

For all $\beta>0$, by the bound (\ref{lower_bound}) in Lemma~\ref
{bound}, we get
%
%e7.5 #&#
%
\begin{eqnarray}
F_N(\beta) \geq\sum_{1\le i<j\le N} \log\cosh
\bigl(\beta J_N(i, j) \bigr)& \geq& C_1\beta^2
\sum_{1\leq i < j \leq N}J_N^2(i, j),
\label{lbpt}
\end{eqnarray}
where $C_1:=\inf_{|x|\le1}\frac{\log\cosh x}{x^2}>0$. To get the
upper bound, %set $r:=\beta_0 \left\Vert J_N\right\Vert=\beta_0$.
%Since $r<1$, there
%exists a constant $C:=C(r)$ such that $-\log(1-\beta_0\lambda_i(J_N))
%\leq\beta_0\lambda_i(J_N)+C\beta_0^2\lambda_i(J_N)^2$, for $1\leq i
%\leq N$. Using this and
we use (\ref{upper_bound}) for $\beta<\frac{1}{\lambda}$
%
%e7.6 #&#
%
\begin{eqnarray}\label{ubpt}
F_N(\beta) &\leq&-\frac{1}{2} \sum_{i=1}^N
\log\bigl(1-\beta\lambda_i(J_N) \bigr)\leq
\frac{C_2\beta^2 }{2} \sum_{i=1}^N
\lambda_i\bigl(J_N^2\bigr)
\nonumber
\\
&\leq& \frac{C_2\beta^2 }{2} \tr\bigl(J_N^2\bigr)
\\
&=& \frac{C_2\beta^2}{2} \sum_{i,j=1}^NJ_N(i,
j)^2, \nonumber
\end{eqnarray}
where $C_2=C_2(\beta):=\sup_{|x|\le\beta}\frac{-\log
(1-x)-x}{x^2}<\infty$ for any $\beta<\frac{1}{\lambda}$, and we use
the fact that $\sum_{i=1}^N \lambda_i(J_N)=0$. The bounds (\ref
{lbpt}) and (\ref{ubpt}) together implies (\ref{maincond}) with
$a_N=\sum_{i,j=1}^NJ_N(i, j)^2$ for $\beta=\beta_0<\frac{1}{\lambda
}$. Therefore, if $\sum_{i,j=1}^NJ_N(i, j)^2\rightarrow\infty$, part
(a) follows by Theorem~\ref{GEN}.

Finally, if $\limsup_{N\rightarrow\infty}\sum_{i,j=1}^NJ_N(i,
j)^2<\infty$, then $F_N(\beta)=O(1)$ for $\beta<\frac{1}{\lambda}$
and by Theorem~\ref{genfixed} part (b) follows.

%\begin{remark} A direct consequence of Corollary~\ref{mplepositive} is
%that the MPLE $\{\hat{\beta}_N\}_{N\geq1}$ is $\sqrt N$-consistent
%for bounded degree graphs. More formally, let $G_N=([N], E(G_N))$ be a
%graph on $[N]:=\{1, 2, \ldots, N\}$ with maximum degree $\Delta$ such
%that $|E(G_N)|\geq C N$ for some constant $C$. Consider the Ising
%model (\ref{pmf}) with $J_N(i, j)=\frac{1}{\Delta}$ for $(i, j)\in
%E(G_N)$ and zero otherwise. Since $|m_i(\tau)|\leq1$, for all $\tau
%\in S_N$, by taking $K=1$, $\frac{1}{\sqrt N}\sum_{i=1}^N |m_i(
%\sigma)| \pmb1\{|m_i(\sigma)|> K\}=0,$ which implies condition (a) of
%Theorem~\ref{GEN}. Moreover, there exists constants $C_1, C_2>0$ such
%that $C_1\sqrt N\leqJ_N\left\Vert_F \leq C_2\sqrt N$, proving that
%the MPLE
%is $\sqrt N$-consistent for bounded degree graphs. This shows that
%MPLE is $\sqrt N$-consistent for lattice graphs reconstructing
%classical results (see \cite{guyon} and the references therein).
%\end{remark}

%s7.2 #&#
\subsection{Proof of Example \protect\ref{ptblock}}
\label{ptblockpf}

It is well known that in the Curie--Weiss model for $\beta>1$ (see
\cite{graph_limits_II}, Example~3.9)
%
%e7.7 #&#
%
\begin{align}
\label{complete_mean_field} \lim_{N\rightarrow\infty}\frac{1}{N}\E
_0e^{\frac{\beta}{2 N}\sum
_{1\le i\ne j\le N}\sigma_i\sigma_j}:=F(
\beta)\in(0, \infty).
\end{align}

Note that
\begin{align*}
\sigma'J_N\sigma={}&\frac{1}{N}\sum
_{1\le i\ne j\le
\frac{N}{2}}\sigma_i\sigma_j+
\frac{1}{\sqrt{N}}\sum_{\frac
{N}{2}<i\ne j\le\frac{N}{2}+\sqrt{N}}\sigma_i
\sigma_j
\\
={}&{\sigma^{(1)}}'A_{N}\sigma^{(1)}+{
\sigma^{(2)}}'B_{N}\sigma^{(2)},
\end{align*}
where $A_N$ is a $N/2\times N/2$ matrix with $A_N(i, j)=1/N$, for $i
\ne j$, and $B_N$ is a $\sqrt N\times\sqrt N$ matrix with $B_N(i,
j)=1/\sqrt N$, for $i \ne j$, and $\sigma=(\sigma^{(1)},\sigma
^{(2)})$. Therefore,
%
%e7.8 #&#
%
\begin{align}
\label{eg:bound1} e^{F_N(\beta)}=\E_0 e^{\frac{\beta}{2} \sigma
'J_N\sigma}=\E
_0e^{\frac{\beta}{2} {\sigma^{(1)}}'A_{N}\sigma^{(1)}}\E_0e^{\frac
{\beta}{2}{\sigma^{(2)}}'B_{N}\sigma^{(2)}}.
\end{align}

Note that $
J_N\Vert =1$ and by \eqref{upper_bound} $F_N(\beta)=O(1)$
for $\beta<1$. Thus, there exists no sequence of consistent estimators
for $\beta\in(0, 1)$ by Theorem~\ref{genfixed}.

For $1<\beta<2$, by \eqref{complete_mean_field}
\begin{align*}
0<\liminf_{N\rightarrow\infty}\frac{1}{\sqrt{N}}F_N(\beta)\le
\limsup_{N\rightarrow\infty}\frac{1}{\sqrt{N}}F_N(\beta)<\infty,
\end{align*}
since $\sigma^{(2)'}B_{N}\sigma^{(2)}$ is the Hamiltonian of a
Curie--Weiss model on size $\sqrt N$.
Moreover, $|m_i(\tau)|\leq1$, for all $1\leq i\leq N$ and $\tau\in
S_N$; so taking $K=1$, $\frac{1}{\sqrt N}\sum_{i=1}^N |m_i(\sigma)|
\pmb1\{|m_i(\sigma)|> K\}=0$, establishing condition (a) of
Theorem~\ref{GEN}. Therefore, the MPLE $\{\hat{\beta}_N\}_{N\geq1}$ is
$N^{1/4}$-consistent for $\beta\in(1,2)$ by Theorem~\ref{GEN}.

Similarly, for $\beta>2$
\begin{align*}
0<\liminf_{N\rightarrow\infty}\frac{1}{N}F_N(\beta)\le
\limsup_{N\rightarrow\infty}\frac{1}{N}F_N(\beta)<\infty,
\end{align*}
and so the MPLE $\{\hat{\beta}_N\}_{N\geq1}$ is $\sqrt N$-consistent.

%s7.3 #&#
\subsection{Proof of Corollary \protect\ref{dregular}}

Note that when the sufficient statistic is of the form (\ref{adjm}),
$|m_i(\tau)|\leq1$, for all $\tau\in S_N$. Therefore, taking $K=1$,
$\frac{d_N}{N}\sum_{i=1}^N |m_i(\sigma)| \pmb1\{|m_i(\sigma)|> K\}
=0$, which implies condition (a) of Theorem~\ref{GEN}. Moreover, in
this case, $\llVert J_N\rrVert =1$, and $\sum_{i, j=1}^N J_N(i, j)^2=N/d_N$.
Therefore, part (a) follows by Corollary~\ref{mplepositive}.

By Corollary~\ref{sc}, to show part (b) it suffices to verify that
condition (\ref{ptN}) holds for all $\beta_0>1$. This is done using
the mean field approximation of Lemma~\ref{meanfield}. By plugging in
the vector $(m, m, \ldots, m)'$ for the vector
$\mathbf{z}$ in the RHS of \eqref{mfieldb}
%
%e7.9 #&#
%
\begin{equation}
F_N(\beta_0)\ge N\sup_{m\in[-1,1]} \biggl\{
\frac{\beta_0
m^2}{2}-I(m) \biggr\}, \label{znbeta0lb}
\end{equation}
where $I(x):=\frac{1}{2}(1+x)\log(1+x)+\frac{1}{2}(1-x)\log(1-x)$
for $x\in[-1,1]$.
Thus, it suffices to show that $\sup_{m\in[-1,1]}g(m)>0$, where
$g(m):=\frac{\beta_0 m^2}{2}-I(m)$. To this end, note that
$g''(0)=\beta_0-1>0$, that is, $m=0$ is not a local maximum of $g$.
This implies the RHS of (\ref{znbeta0lb}) is positive, thus verifying
condition (\ref{ptN}).

%s7.4 #&#
\subsection{Proof of Corollary \protect\ref{erdosrenyirandom}}

Let $d_i$ be the degree of the vertex $i$ in $G_N$, for $1\leq i \leq
N$. Then $|m_i(\tau)|\leq\frac{d_i}{N p(N)}$, for all $\tau\in
S_N$. In the regime $\frac{\log N}{N}\ll p(N) \leq1$, the maximum
degree $\Delta=\max_{i \in V(G_N)}d_i=Np(N)(1+o(1))$ with high
probability \cite{ks}. Therefore, $|m_i(\tau)|\leq1+o(1)$ for all
$1\leq i \leq N$, and by taking $K\geq2$ it follows that $p(N)\sum
_{i=1}^N |m_i(\sigma)| \pmb1\{|m_i(\sigma)|> K\}=0$, with high
probability. This implies condition (a) of Theorem~\ref{GEN}.
%In this case $p(N)\sum_{i,j=1}^N J_N(i, j)^2=\frac{2}{N^2p(N)}|E(G_N)|
%\pto1$. This verifies condition (b) in Theorem~\ref{GEN}.
%Also, as in (\ref{lbpt})
%\begin{equation}
%p(N) F_N(\beta) \geq C_1\beta^2\cdot p(N) \sum_{1\leq i < j \leq
%N}J_N^2(i, j)\pto\frac{C_1\beta^2}{2}.
%\label{erlb}
%\end{equation}

Moreover, for $\frac{\log N}{N}\ll p(N) \leq1$, $\llVert J_N\rrVert
=1+o(1)$ with
high probability \cite{ks}, and
\[
p(N)\sum_{i,j=1}^N J_N(i,
j)^2=\frac{2}{N^2p(N)}\bigl|E(G_N)\bigr|\pto1,
\]
and part (a) follows from Corollary~\ref{mplepositive}.

%Therefore, as in (\ref{ubpt}), for $\beta<1$
%\begin{eqnarray}
%p(N)F_N(\beta)& \leq& C_2(\beta) \beta^2 \cdot p(N) \sum_{1\leq i \ne
%j \leq N}J_N(i, j)^2 \pto C_2(\beta) \beta^2.
%\label{erub}
%\end{eqnarray}
%The bounds (\ref{erlb}) and (\ref{erub}) imply (\ref{maincond}) with
%$a_N=1/p_N$. Therefore, if $p(N)\ll1$, then by Theorem~\ref{GEN},
%part (a) follows. For $p_N=O(1)$, $F_N(\beta)=O(1)$ for $\beta<1$ (by (
%\ref{erub})), and by Theorem~\ref{genfixed} $\hat\beta_N$ is
%inconsistent.
%

To prove part (b), we use the mean field approximation as in
Corollary~\ref{dregular}. By plugging in the vector $(m, m, \ldots,
m)'$ for
the vector
$\mathbf{z}$ in the RHS of \eqref{mfieldb}, we get
\[
F_N(\beta_0)\ge N\sup_{m\in[-1,1]} \biggl\{
\frac{\beta_0
m^2|E(G_N)|}{N^2p(N)}-I(m) \biggr\}.
\]
Condition \eqref{ptN} follows by arguments similar to those in
Corollary~\ref{dregular} and the fact $\frac
{2|E(G_N)|}{N^2p(N)}\stackrel{\sP}{\rightarrow}1$.

%Combining (\ref{upper}) and (\ref{lower}) it follows that $0<
%\liminf_{N\rightarrow\infty} \frac{1}{N}F_N(\beta_0)\leq\limsup_{N
%\rightarrow\infty} \frac{1}{N} F_N(\beta_0) <\infty$, verifying
%condition (\ref{ptN}).

%s8 #&#
\section{Proofs of Theorems \protect\ref{dense_testing} and~\protect\ref{dense}}
\label{pfdense}

In this section, we show the existence of a untestable/testable
threshold in Ising models on converging sequence of dense graphs, and
compute the distribution and asymptotic power of the most powerful
test, before the phase transition.

%s8.1 #&#
\subsection{Proof of Theorem \protect\ref{dense_testing}}
\label{threshold}

If $G_N$ converges to $W$, then $\frac{1}{N}\llVert A(G_N)\rrVert $
converges to
the operator norm of $\llVert W\rrVert $ (see (\ref{eq:T})). Moreover,
\[
\frac{1}{N^2}\sum_{i=1}^N
\lambda_i\bigl(A(G_N)^2\bigr)\rightarrow
t(C_2, W),
\]
and part (a) follows by Corollary~\ref{mplepositive}.

%To show (a) it suffices to verify $\limsup_{N\rightarrow\infty} F_N(
%\beta)< \infty$, for $\beta<\frac{1}{\left\Vert W\right\Vert}$ (by
%Proposition~\ref{TESTING}). To this end, by (\ref{upper_bound})
%\begin{eqnarray}
%F_N(\beta)\leq-\frac{1}{2} \sum_{i=1}^N \log\left(1-\beta
%\lambda_i(W_N)\right)&\leq&
%C _2(\beta)\beta^2 \cdot\frac{1}{N^2}\sum_{i=1}^N\lambda_i(A(G_N)^2)
%\nonumber\\
%&\leq& C_2(\beta) \beta^2 \cdot t(C_2, W_N)\nonumber\\
%&\rightarrow& C_2(\beta) \beta^2 \cdot t(C_2, W),
%\label{denseupper}
%\end{eqnarray}
%where the last step uses (\ref{lambda2}).

%To prove (b) it suffices to show the stronger conclusion that the MPLE
%$\{\hat\beta_N\}_{N\geq1}$ is $\sqrt N$ consistent for $\beta_1>
%1/\left\Vert W\left\Vert$.
We now show (b). From \cite{graph_limits_II}, Theoem 2.14, when $G_N$ converges to $W$,
then $\lim_{N\rightarrow\infty}\frac{1}{N}F_N(\beta)=\sE(W, \beta
)$, where
%
%e8.1 #&#
%
\begin{equation}
\label{denseopt} \sE(W, \beta):=\sup_{m:[0,1]\mapsto[-1,1]} \biggl
\{
\frac{\beta
}{2}
\int_{[0,1]^2}m(x)m(y)W(x,y)\,\mathrm{d} x\,\mathrm{d} y-
\int_0^1I\bigl(m(x)\bigr)\,\mathrm{d} x \biggr\},
\end{equation}
and $I(x)=\frac{1}{2} \{(1+x)\log(1+x)+(1-x)\log(1-x) \}$
as in Corollary~\ref{dregular}.
By Corollary~\ref{sc}, it enough to show that $\sE(W, \beta)>0$,
for $\beta> \frac{1}{\llVert W\rrVert }$.

To this end, let $v_1(x)$ to be the eigenvector corresponding to the
eigenvalue $\lambda=\llVert W\rrVert $. Then $\vert\lambda
v_1(x)\vert=\vert\int_0^1 W(x,y)v_1(y)\,\mathrm{d} y\vert\le1$,
and $\sup_{x\in
[0,1]}|v_1(x)|<\infty$. Thus, there exists $\delta>0$ such that for
$z\in(-\delta,\delta)$ we have $\sup_{x\in[0,1]}|zv_1(x)|\le1$, and
\begin{align*}
\sE(W, \beta) \ge{}& \sup_{|z|<\delta} \biggl\{\frac{\beta}{2}
z^2
\int_{[0, 1]^2}v_1(x)v_1(y)W(x,y)\,\mathrm{d} x
\,\mathrm{d} y-
\int_0^1I\bigl(z v_1(x)\bigr)
\,\mathrm{d} x \biggr\}
\\
={}&\sup_{|z|<\delta} \biggl\{\frac{\beta z^2\lambda}{2}
\int_0^1 v_1(x)^2\,\mathrm{d}x-
\int_0^1I\bigl(z v_1(x)\bigr)\,\mathrm{d} x \biggr\}.
\end{align*}

Setting $h(z):=\frac{\beta z^2\lambda}{2}\int_0^1 v_1(x)^2\,\mathrm{d}
x-\int_0^1I(z v_1(x))\,\mathrm{d} x$ it suffices to show that $z=0$ is
not a point of local maxima of the function $h$. This follows on noting
that $h''(0)=(\beta\lambda-1) \int_0^1 v_1(x)^2 \,\mathrm{d} x>0$.

%s8.2 #&#
\subsection{Proof of Theorem \protect\ref{dense}}

By Lemma~\ref{lem:exists} (see Appendix \ref{pfpartition}), the
limiting distribution (\ref{limitdense}) is well defined.

The following proposition (proved in Appendix \ref{pfpartition}) gives
the limit of the log-partition function, for a converging sequence of
dense graphs, for $\beta< \frac{1}{\llVert W\rrVert }$.

%pr8.1 #&#
%
\begin{ppn}\label{PARTITION}
Let $\{G_N\}_{N\geq1}$ be a sequence of simple graphs converging in cut-metric to $W \in\sW$, such
that $\int_{[0, 1]^2}W(x, y)^2\,\mathrm{d} x\,\mathrm{d} y>0$. Then
for any
$0< \beta<\frac{1}{\llVert W\rrVert }$
%
%e8.2 #&#
%
\begin{equation}
\label{znlimit} \lim_{N\rightarrow\infty}F_N(\beta)=-
\frac{1}{2} \sum_{i=1}^\infty\bigl\{
\log\bigl(1- \beta\lambda_i(W) \bigr)- \beta\lambda_i(W)
\bigr\}.
\end{equation}
\end{ppn}

The above proposition can be used to complete the proof of Theorem~\ref
{dense_testing} as follows: Fix $\delta>0$ such that $\beta+\delta
<\frac{1}{\llVert W\rrVert }$. Then for any $t\in(-\beta,\delta)$,
%
%e8.3 #&#
%
\begin{eqnarray}
\E_{\beta}\exp\biggl\{\frac{t}{2}\cdot\frac{1}{N}\sigma
'W_N\sigma\biggr\}&=& \exp\bigl\{F_N(
\beta+t)-F_N(\beta)\bigr\}
\nonumber
\\[-8pt]
\\[-8pt]
\nonumber
&\rightarrow& \prod_{i=1}^\infty
\frac{e^{-\frac{1}{2}t\lambda_i(W)
}}{\sqrt{1-\frac{t\lambda_i(W)}{1-\beta\lambda_i(W)}}},
\end{eqnarray}
by Proposition~\ref{PARTITION}.

By Lemma~\ref{lem:exists} the RHS above is the MGF of the random
variable $\frac{Q_{\beta, W}}{2}$ defined in (\ref{limitdense}).

\begin{appendix}

%sA #&#
\section{Proofs of technical lemmas}
\label{pfgenlemmas}

In the appendix we prove the lemmas used in the proof of Theorem~\ref
{GEN}. The rest of the section is organized as follows: Appendix \ref
{pflem1} contains the proof of Lemma~\ref{lem1}. The proofs of Lemmas
\ref{2moment} and \ref{derivative} are given in Appendices~\ref
{pflem2moment} and~\ref{pflemderivative}, respectively.

%sA.1 #&#
\subsection{Proof of Lemma \protect\ref{lem1}}
\label{pflem1}

By \eqref{maincond} there exists $\delta\in(0,\beta_0/2)$ such that
$\liminf_{N\rightarrow\infty}\frac{1}{a_N}F_N(\beta_0-\delta)>0$.
By the monotonicity of $F_N'(\cdot)$,
\[
F_N(\beta_0-\delta)=
\int_0^{\beta_0-\delta}F_N'(t)\,\mathrm{d}t\le(\beta_0-\delta)F_N'(
\beta_0-\delta)\le\beta_0F_N'(
\beta_0-\delta),
\]
it follows that $\liminf_{N\rightarrow\infty}\frac
{1}{a_N}F_N'(\beta_0-\delta)>0$.
Thus, for any $\varepsilon>0$
\begin{align*}
&\P_{\beta_0}\bigl(H_N(\sigma)<\varepsilon a_N
\bigr)=\P_{\beta_0}\bigl(e^{-\frac
{1}{2}\delta H_N(\sigma)}>e^{-\frac{1}{2}\delta\varepsilon a_N}\bigr
) \le
e^{\frac{1}{2}\delta\varepsilon a_N+F_N(\beta_0-\delta)-F_N(\beta_0)},
\end{align*}
which, on taking logarithms, implies that
\begin{align*}
\log\P_{\beta_0}\bigl(H_N(\sigma)<\varepsilon a_N
\bigr)\le\frac{\varepsilon
\delta a_N}{2}-
\int_{\beta_0-\delta}^{\beta_0}F_N'(t)
\,\mathrm{d}t\le\frac{\varepsilon\delta a_N}{2}-F_N'(
\beta_0-\delta)\delta.
\end{align*}
Dividing both sides by $a_N$ and taking limits as $N\rightarrow\infty
$ followed by $\varepsilon\rightarrow0$ we have
\[
\lim_{\varepsilon\rightarrow0}\limsup_{N\rightarrow\infty}\frac
{1}{a_N}
\log\P_{\beta_0}\bigl(H_N(\sigma)<\varepsilon a_N
\bigr)\le-\liminf_{N\rightarrow\infty}\frac{1}{a_N}{F_N'(
\beta_0-\delta)}<0,
\]
thus completing the proof of (a).

To show (b), again invoking \eqref{maincond} there exists $\delta>0$
such that
$\limsup_{N\rightarrow\infty}F_N(\beta_0+2\delta)<\infty$. Since
\[
F_N(\beta_0+2\delta)=
\int_0^{\beta_0+2\delta}F_N'(t)\,\mathrm{d}t\ge\delta F_N'(\beta
_0+\delta),
\]
it follows that $\limsup_{N\rightarrow\infty}\frac
{1}{a_N}F_N'(\beta_0+\delta)<\infty$. Thus, for any $K<\infty$
\begin{align*}
&\P\bigl(H_N(\sigma)>K a_N\bigr)=\P\bigl(e^{\frac{1}{2}\delta
H_N(\sigma
)}>e^{\frac{1}{2}\delta Ka_N}
\bigr)\le e^{-\frac{1}{2}\delta K
a_N+F_N(\beta_0+\delta)-F_N(\beta_0)}.
\end{align*}
Taking logarithm on both sides,
\begin{align*}
\log\P\bigl(H_N(\sigma)>K a_N\bigr)& \le-
\frac{\delta K a_N}{2}+
\int_{\beta
_0}^{\beta_0+\delta}F_N'(t)
\,\mathrm{d}t
\nonumber
\\
& \le-\frac{\delta Ka_N}{2}+F_N'(\beta_0+
\delta),
\end{align*}
from which dividing by $a_N$ and taking limits as $N\rightarrow\infty
$ followed by $K\rightarrow\infty$ gives
\[
\lim_{K\rightarrow\infty}\limsup_{N\rightarrow\infty}\frac
{1}{a_N}\log
\P\bigl(H_N(\sigma)>K a_N\bigr)=-\infty,
\]
thus proving part (b).

%sA.2 #&#
\subsection{Proof of Lemma \protect\ref{2moment}}
\label{pflem2moment}

We begin with a technical estimate which will be needed to bound the
second moment of $L_\sigma(\beta_0)$.

%leA.1 #&#
%
\begin{lem}\label{lem0}
Under assumption \eqref{maincond} and $m_i(\sigma)$ as defined in
(\ref{mi}),
\[
\limsup_{N\rightarrow\infty}\frac{1}{a_N}\E_{\beta_0}\sum
_{i=1}^Nm_i(\sigma)\tanh\bigl(
\beta_0m_i(\sigma)\bigr)<\infty.
\]
\end{lem}

\begin{proof}
By \eqref{maincond} there exists $\delta>0$ such that
$\limsup_{N\rightarrow\infty}\frac{1}{a_N}F_N(\beta_0+\delta
)<\infty$. Therefore, $F_N(\beta_0+\delta)=\int_0^{\beta_0+\delta
}F_N'(t)\,\mathrm{d}t\ge\delta F_N'(\beta_0)$,
and so %$\limsup_{N\rightarrow\infty}\frac{1}{a_N}F_N'(\beta_0+)<
%\infty$. By monotonicity,
%
%eA.1 #&#
%
\begin{equation}
\frac{1}{a_N}\limsup_{N\rightarrow\infty}Z'_N(
\beta_0)<\infty. \label{derivative_finite}
\end{equation}

Now, observe that $m_i(\sigma)$ does not depend on $\sigma_i$, and
$\E_{\beta_0}(\sigma_i|(\sigma_j)_{j\ne i})=\tanh(\beta
_0m_i(\sigma))$. Since
\begin{align*}
2F_N'(\beta_0)={}&\E_{\beta_0}
H_N(\sigma)=\E_{\beta_0} \Biggl(\sum
_{i=1}^N \sigma_i m_i(
\sigma) \Biggr)
\nonumber
\\
={}&\E_{\beta_0} \Biggl(\sum_{i=1}^N
m_i(\sigma)\tanh\bigl(\beta_0 m_i(\sigma)
\bigr) \Biggr),
\end{align*}
the result follows from \eqref{derivative_finite}.
\end{proof}

The above lemma will be used to complete the proof of Lemma~\ref
{2moment}. To this end, for $1\leq j \leq N$ and $\tau\in S_N$, let
\[
\tau^{(j)}:=(\tau_1, \ldots, \tau_{j-1}, -
\tau_j, \tau_{j+1}, \ldots, \tau_N)
\]
and
%
%eA.2 #&#
%
\begin{equation}
p_j(\tau)=\frac{e^{-\beta_0 \tau_j m_j(\tau)}}{e^{\beta_0 \tau_j
m_j(\tau)}+e^{-\beta_0 \tau_j m_j(\tau)}}. \label{probj}
\end{equation}

From equation (10) of Chatterjee \cite{Chatterjee} it follows that
%
%eA.3 #&#
%
\begin{equation}
\E_\beta\bigl(L_\sigma(\beta_0)^2
\bigr)=\frac{1}{N}\E_\beta\sum_{j=1}^N
\bigl(L_{\sigma}(\beta_0)-L_{\sigma^{(j)}}(
\beta_0) \bigr)m_j(\sigma)\sigma_jp_j(
\sigma).
\end{equation}
Setting $r(x):=x\tanh(\beta_0x)$, note that
\[
L_\sigma(\beta_0)-L_{\sigma^{(j)}}(\beta_0)=
\frac{2 m_j(\sigma
)\sigma_j}{N}+\frac{1}{N}\sum_{i=1}^N
\bigl\{r\bigl(m_i\bigl(\sigma^{(j)}\bigr)\bigr)-r
\bigl(m_i(\sigma)\bigr)\bigr\}.
\]
Now, by a second order Taylor expansion,
%
%eA.4 #&#
%
\begin{align}
\E_\beta\bigl(L_\sigma(\beta_0)^2
\bigr)=\frac{a_N}{N^2} (T_1+T_2+T_3 ),
\end{align}
where
\[
T_1=\frac{2}{a_N}\sum_{j=1}^Nm_j(
\sigma)^2p_j(\sigma), \qquad T_2=-\frac{2}{a_N}
\sum_{i,j=1}^NJ_N(i,j)
r'\bigl(m_i(\sigma)\bigr)m_j(\sigma)
p_j(\sigma)
\]
and
\[
T_3=\frac{2}{a_N}\sum_{i,j=1}^N
r''\bigl(\theta_{ij}(\sigma)
\bigr)J_N(i,j)^2m_j(\sigma)
\sigma_jp_j(\sigma),
\]
for some $\theta_{ij}(\sigma)$ in the interval $[m_i(\sigma^{(j)}),
m_i(\sigma)]$. Therefore, to prove the lemma, it suffices to control
these three terms.

To control $T_1$, note that
\[
\E_{\beta_0}\bigl(p_j(\sigma)|(\sigma_i)_{i\ne j}
\bigr)=\frac{2}{e^{\beta
_0m_j(\sigma)}+e^{-\beta_0m_j(\sigma)}}=\frac{1}{2}\operatorname{sech}^2\bigl(
\beta_0m_j(\sigma)\bigr),
\]
and $x^2\operatorname{sech}^2(\beta_0 x) \leq M_1 x\tanh(\beta_0x)$ for all
$x\in\R$ for some $M_1=M_1(\beta_0)<\infty$, which gives
%
%eA.5 #&#
%
\begin{eqnarray}
\label{T1} \E_{\beta_0}T_1 &=&\frac{1}{a_N}
\E_{\beta_0}\sum_{j=1}^Nm_j(
\sigma)^2 \operatorname{sech}^2\bigl(\beta_0
m_j(\sigma)\bigr)
\nonumber
\\[-8pt]
\\[-8pt]
\nonumber
& \leq&\frac{M_1}{a_N} \E_{\beta
_0}\sum_{j=1}^Nm_j(
\sigma) \tanh\bigl(\beta_0 m_j(\sigma)\bigr),
\end{eqnarray}
which is finite as $N\rightarrow\infty$ by an application of
Lemma~\ref{lem0}.

Now, let us bound $T_2$. By the Cauchy--Schwarz inequality,
\begin{align*}
|T_2|&\le\frac{2}{a_N} \Biggl\{\sum
_{i=1}^Nr'\bigl(m_i(
\sigma)\bigr)^2 \Biggr\}^{1/2} \Biggl\{ \sum
_{i=1}^N \Biggl(\sum_{j=1}^N
J_N(i, j)m_j(\sigma)\sigma_jp_j(
\sigma) \Biggr)^2 \Biggr\}^{1/2}
\\
& \le\frac{2\llVert J_N\rrVert }{a_N} \Biggl\{\sum_{i=1}^Nr'
\bigl(m_i(\sigma)\bigr)^2 \Biggr\}^{1/2} \Biggl
\{ \sum_{j=1}^N m_j(
\sigma)^2p_j(\sigma)^2 \Biggr
\}^{1/2}.
\end{align*}
Taking expectation on both sides above and using Cauchy--Schwarz
inequality again
%
%eA.6 #&#
%
\begin{align}
\label{t2cs} \E_{\beta_0}|T_2|\leq\frac{2\llVert J_N\rrVert
}{a_N} \Biggl\{
\E_{\beta_0}\sum_{i=1}^N
r'\bigl(m_i(\sigma)\bigr)^2 \cdot
\E_{\beta_0}\sum_{j=1}^N
m_j(\sigma)^2p_j(\sigma) \Biggr
\}^{1/2}.
\end{align}
Now, since $r'(x)^2=\{\tanh(\beta_0x)+\beta_0x \operatorname{sech}^2(\beta
_0x)\}^2\leq M_2 x\tanh(\beta_0x)$, for some constant $M_2=M_2(\beta
_0)$, by Lemma~\ref{lem0}
\[
\limsup_{N\rightarrow\infty}\frac{1}{a_N}\E_{\beta_0}\sum
_{i=1}^Nr'\bigl(m_i(
\sigma)\bigr)^2<\infty.
\]
Using this along with \eqref{T1} in (\ref{t2cs}) gives $\limsup
_{N\rightarrow\infty}\E_{\beta_0}|T_2|<\infty$.

It remains to bound $T_3$. Since $M_3=M_3(\beta_0):=\sup_{x\in\R
}|r''(x)|<\infty$, we have
%
%eA.7 #&#
%
\begin{eqnarray}\label{T3bound}
|T_3|&\leq&\frac{2M_3}{a_N}\sum_{i,j=1}^NJ_N(i,j)^2\bigl|m_j(
\sigma)\bigr|p_j(\sigma)
\nonumber
\\
&\leq& \frac{2M_3}{a_N} \Biggl\{\sum_{j=1}^N
\Biggl(\sum_{i=1}^N J_N(i,
j)^2 \Biggr)^2 \Biggr\}^{1/2} \Biggl\{\sum
_{j=1}^N m_j(\sigma
)^2p_j(\sigma) \Biggr\}^{1/2}
\\
&\leq& \frac{2M_3 \llVert J_N\rrVert }{a_N} \Biggl\{\sum_{i, j=1}^N
J_N(i, j)^2 \Biggr\}^{1/2} \Biggl\{\sum
_{j=1}^N m_j(\sigma)^2p_j(
\sigma) \Biggr\}^{1/2}, \nonumber
\end{eqnarray}
where the last step uses $\sum_{i=1}^N J_N(i,j)^2=\llVert J_N
e_j\rrVert ^2\le
\llVert J_N\rrVert ^2$. Finally, taking expectations on both sides in
(\ref
{T3bound}), and using condition (b) on the first term, and (\ref{T1})
on the second term, gives $\limsup_{N\rightarrow\infty}\E_{\beta
_0}|T_3|<\infty$.

\subsection{Proof of Lemma \protect\ref{derivative}}
\label{pflemderivative}
Fixing $\delta>0$ by Lemma~\ref{lem1}(a) there exists $\varepsilon
=\varepsilon(\delta)>0$ such that
%
%eA.8 #&#
%
\begin{equation}
\P_{\beta_0}\bigl(H_N(\sigma)<3\varepsilon\beta_0
a_N\bigr)\le\delta, \label{c00}
\end{equation}
for $N$ large enough. Also, using Lemma~\ref{2moment} and Chebyshev's
inequality, for $K_1=K_1(\delta):=\sqrt{\frac{C(\beta_0)}{\delta
}}$ we have
%
%eA.9 #&#
%
\begin{eqnarray}
\P_{\beta_0}\bigl(\bigl|L_\sigma(\beta_0)\bigr|>K_1
\sqrt{a_N}/N\bigr)\le\frac
{N^2}{a_N}\E_{\beta_0}
L_\sigma(\beta_0)^2 \le\delta. \label{c1}
\end{eqnarray}

Moreover, by condition (a) in Theorem~\ref{GEN} there exists
$K_2=K_2(\delta)<\infty$ such that for all $N$ large enough we have
\[
\E_{\beta_0}\sum_{i=1}^N\bigl|m_j(
\sigma)\bigr|\pmb1\bigl\{\bigl|m_j(\sigma)\bigr|>K_2\bigr\} \le
\varepsilon\delta\beta_0 a_N
\]
and so by Markov's inequality
%
%eA.10 #&#
%
\begin{eqnarray}
\label{c0} &&\P_{\beta_0} \Biggl(\sum_{i=1}^N\bigl|m_j(
\sigma)\bigr|\pmb1\bigl\{\bigl|m_j(\sigma)\bigr| >K_2\bigr\}>
\varepsilon\beta_0 a_N \Biggr)
\nonumber
\\[-8pt]
\\[-8pt]
\nonumber
&&\qquad\quad \le \frac{\E_{\beta_0}\sum_{i=1}^N|m_j(\sigma)|\pmb1\{
|m_j(\sigma)|>K_2\}}{\varepsilon\beta_0 a_N}\le\delta.
\end{eqnarray}

Defining
\begin{align*}
A_{N}(\delta):={}& \Biggl\{\sigma\in S_N: H_N(
\sigma)\ge3\varepsilon\beta_0 a_N, \bigl|L_\sigma(
\beta_0)\bigr|\le K_1\frac{\sqrt{a_N}}{N},
\\
&\sum_{i=1}^N\bigl|m_j(\sigma)\bigr|
\pmb1\bigl\{|m_j(\sigma)|>K_2\bigr\}> \varepsilon
\beta_0 a_N \Biggr\},
\nonumber
\end{align*}
we have $\P_{\beta_0}(A_N(\delta))\ge1-3\delta$, for $N$ large
enough (by combining (\ref{c00}), (\ref{c1}), and (\ref{c0})).

Now, on the set $A_N(\delta)$ using the bounds $\tanh x\leq x$ on
$x\leq K_2$, and $\tanh x\leq1$ on $x> K_2$,
\begin{align*}
\beta_0\sum_{i=1}^N
m_i(\sigma)^2 \pmb1\bigl\{\bigl|m_i(\sigma)\bigr|\le
K_2\bigr\} +\varepsilon\beta_0 a_N \ge{}&\sum
_{i=1}^N m_i(\sigma)\tanh
\bigl(\beta_0 m_i(\sigma)\bigr)
\\
%=&-\sum_{i=1}^N (\sigma_i-\tanh(\beta_0 m_i(\sigma)))m_i(\sigma)+H_N(
%\sigma)\\
={}&H_N(\sigma)-NL_\sigma(
\beta_0)
\nonumber
\\
\ge{}& 3\varepsilon\beta_0 a_N-K_1
\sqrt{a_N}.
\end{align*}
Thus, on the set $A_{N}(\delta)$,
\[
\sum_{i=1}^N m_i(
\sigma)^2 \pmb1\bigl\{\bigl|m_i(\sigma)\bigr|\le K_2
\bigr\}\ge2\varepsilon a_N-\frac{K_1}{\beta_0}\sqrt{a_N}>
\varepsilon a_N
\]
for all $N$ large, completing the proof.

%sA #&#
\section{Proof of Lemma \protect\ref{TESTING}}
\label{pftesting}

For every $N\geq1$, let $(\sX_N, \cF_N)$ be a measure space and $\P
_N$ and $\Q_N$ two distributions on this measure space. Recall the
definition of Kullback--Leibler divergence $D(\Q_N\Vert \P_N)$ from
\eqref
{kldivg}, and consider the problem of testing $\P_N$ versus $\Q_N$
such that condition \eqref{divlemma} holds. Since $D(\Q_N\Vert \P
_N)=\mathbb{E}_{\Q_N}L_N$, by assumption (\ref{divlemma})
%
%eA.1 #&#
%
\begin{equation}
0\leq\mathbb{E}_{\Q_N}L_N=\mathbb{E}_{\Q_N}L_N^+-
\mathbb{E}_{\Q
_N}L_N^-\leq M_1, \label{likelihood}
\end{equation}
for some $M_1<\infty$ and all large $N$. Also, there exists
$M_2<\infty$ such that $\mathbb{E}_{Q_N}L_N^-\leq M_2$, for all~$N$.
To see this, note that
%
%eA.2 #&#
%
\begin{align}\label{ln}
\mathbb{E}_{\Q_N}L_N^{-}={}&-\sum
_{s=1}^\infty\mathbb{E}_{\Q
_N}L_N
\pmb1\{-s\leq L_N< -s+1\}
\nonumber
\\
\leq{}&\sum_{s=1}^\infty s e^{-(s-1)}
\P_N(-s\le L_N<-s+1)
\\
\leq{}&\sum_{s=1}^\infty se^{-(s-1)}:=M_2<
\infty. \nonumber
\end{align}

Hence, by (\ref{likelihood}) and (\ref{ln}), $\mathbb{E}_{\Q
_N}|L_N|=\mathbb{E}_{\Q_N}L_N^++\mathbb{E}_{\Q_N}L_N^-\leq
M_1+2M_2=:M<\infty$. Therefore, by Markov's inequality, for any
$\varepsilon>0$
\[
\Q_N\bigl(|L_N|>M/\varepsilon\bigr)\le\frac{\varepsilon}{M}
\E_{\Q_N} \bigl(|L_N|\bigr)\le\varepsilon.
\]
Now, suppose there exists a sequence of test functions $\phi_N$ such
that $\E_{\P_N}\phi_N\rightarrow0$. Then
\begin{align*}
\E_{\Q_N}\phi_N\le\Q_N\bigl( |L_N|>M/
\varepsilon\bigr)+\E_{\Q_N} \bigl(\phi_N\pmb1\bigl\{|L_N|
\le M/\varepsilon\bigr\} \bigr)\le\varepsilon+e^{M/\varepsilon}\E_{\P_N}
\phi_N.
\end{align*}
Taking limits on both sides gives, $\limsup_{N\rightarrow\infty}\E
_{\Q_N}\phi_N\le\varepsilon$. Since $\varepsilon>0$ is arbitrary\break
$\lim_{N\rightarrow\infty}\E_{\Q_N}\phi_N=0$, that is, $\phi_N$
is not a consistent sequence of test functions.

%sA #&#
\section{The mean-field approximation}

A standard technique to derive a lower bound on the log-partition
function is the mean-field approximation (refer to \cite
{dembo_montanari} for details). Here, we give a short proof for the
sake of completeness.

%leA.1 #&#
%
\begin{lem}\label{meanfield} Consider the family of probability
distributions on $S_N$ given by (\ref{pmf}). Then for any matrix
%
%eA.1 #&#
%
\begin{equation}
F_N(\beta)\ge\sup_{\it{\mathbf{z}}\in[-1,1]^N} \Biggl\{
\frac
{\beta}{2} \it{\mathbf{z}}'J_N\it{\mathbf{z}}-
\sum_{i=1}^NI(z_i) \Biggr\},
\label{mfieldb}
\end{equation}
where $I(x)=\frac{1}{2} [(1+x)\log(1+x)+(1-x)\log(1-x) ]$
for $x\in[-1,1]$.
\end{lem}

\begin{proof}
Let $D(\cdot\Vert \cdot)$ be the Kullback--Leibler divergence
between two
probability measures. By a direction computation, for any probability
mass function $\nu$ on $S_N=[-1, 1]^N$ we have
\begin{align*}
D(\nu\Vert\P_\beta)=F_N(\beta)+N\log2+\E_{\nu}
\log\nu(\sigma)-\frac{1}{2}\E_{\nu}H_N(\sigma).
\end{align*}
Now, since $D(\nu\Vert \P_\beta)\ge0$ we have
\[
F_N(\beta)\ge\frac{\beta}{2}\E_{\nu}H_N(
\sigma)-\E_{\nu}\log\nu(\sigma)-N\log2.
\]
One can obtain a lower bound on $F_N(\be)$ by taking supremum in LHS
over product measures, that is $\nu(\sigma)=\prod_{i=1}^N \nu
_i(\sigma_i)$. Hence, setting $z_i=\E_{\nu_i}\sigma=\nu_i(1)-\nu
_i(-1)\in[-1,1]$, the bound in \eqref{mfieldb} follows.
\end{proof}

%sA #&#
\section{Proof of Proposition \protect\ref{PARTITION}}
\label{pfpartition}

We begin by deriving the MGF of the limiting distribution \eqref
{limitdense}. The proof involves straightforward calculations using the
MGF of the chi-squared distribution, similar to \cite{BDM}, Proposition~7.1.

%leA.1 #&#
%
\begin{lem}\label{lem:exists}
Let $\{a_i\}_{i\ge1},\{b_i\}_{i\ge1}$ be a sequence of real numbers
such that $\sum_{i=1}^\infty a_i^2<\infty$ and $\sum_{i=1}^\infty
(a_i-b_i)=\mu$ for some finite real number $\mu$. Suppose $\xi_1,\xi
_2, \ldots$ be i.i.d. $\chi_1^2$ random variables.
\begin{enumerate}[(a)]
\item[(a)] Then the sum $S:=\frac{1}{2}\sum_{i=1}^\infty(a_i\xi_i-b_i)$
converges almost surely and in $L^1$.

\item[(b)] Moroever, if $M:=\sup_{i\ge1}|a_i| < \infty$, then for
$0<t<\frac{1}{ M}$,
%
%eA.1 #&#
%
\begin{equation}
\E e^{\frac{1}{2}t\sum_{i=1}^\infty(a_i\xi_i-b_i)}=\prod
_{i=1}^\infty
\frac{e^{-\frac{1}{2}t b_i}}{\sqrt{1-ta_i}}. \label{mgf}
\end{equation}
\end{enumerate}
\end{lem}

\begin{proof}By defining $S_N:=\frac{1}{2}\sum_{i=1}^N (a_i\xi_i-b_i)$
and $\sF_N:=\sigma(\{\xi_j\}_{j=1}^N)$, it follows that $(S_N, \sF
_N)$ is a martingale, with
\[
\limsup_{N}\E S_N^2=
\frac{1}{4} \Biggl(\mu^2+\sum_{j=1}^\infty
a_j^2 \Biggr)<\infty,
\]
and so $S_N$ converges almost surely and in $L^1$ \cite{durrett}.

To compute the moment generating function of $S$, first note that
$e^{tS_N}\stackrel{\sP}{\rightarrow} e^{tS}$. Thus if the collection
of random variables $\{e^{tS_N}\}$ is uniformly integrable, then we have
\begin{align*}
\E e^{tS}=\lim_{N\rightarrow\infty}\E e^{tS_N}=\lim
_{N\rightarrow
\infty}\prod_{i=1}^N
\frac{e^{-\frac{1}{2}tb_i}}{\sqrt
{1-ta_i}}=\prod_{i=1}^\infty
\frac{e^{-\frac{1}{2}tb_i}}{\sqrt{1-ta_i}},
\end{align*}
thus completing the proof of the lemma. It thus remains to prove
uniform integrability, for which it suffices to show that for some
$\delta>0$ we have
$\limsup_{N\rightarrow\infty}\E e^{(t+\delta)S_N}<\infty$. Since
$t<\frac{1}{M}$ there exists $\delta>0$ such that $t+\delta<\frac
{1}{M}$. For this $\delta$ setting $t':=t+\delta$ we have
%
%eA.2 #&#
%
\begin{align}
\label{eq:bound} \log\E e^{t'S_N}=\frac{1}{2}\sum
_{i=1}^N\bigl\{-t'b_i-\log
\bigl(1-t'a_i\bigr)\bigr\}.
\end{align}
Now setting $C:=\sup_{|x|\le t'M}\frac{-\log(1-x)-x}{x^2}<\infty$
we have
$-\log(1-x)-x\le Cx^2$ for $|x|< t'M$, and so the RHS of \eqref
{eq:bound} can be bounded by
$
\frac{1}{2}\sum_{i=1}^N\{t'(a_i-b_i)+Ct'^2a_i^2\}$,
which converges to $e^{t'\mu+Ct'^2\sum_{i=1}^\infty a_i^2}$.
Therefore, $e^{t S_N}$ is uniformly integrable, thus completing the
proof of the lemma.
\end{proof}

The above lemma can be used to complete the proof of Proposition~\ref
{PARTITION}. To this end, let $W_N:=A(G_N)$. Then, by \cite{BDM}, Theorem~1.4, it follows that 
\[
\frac{1}{N}{\sigma}'W_N{\sigma} \dto
\sum_{i=1}^\infty\lambda_i(W) (
\xi_i-1),
\]
where $\xi_1,\xi_2,\ldots,$ are i.i.d. $\chi_1^2$ random variables. Thus,
%
%eA.3 #&#
%
\begin{align}
\label{uniform_integrable} \exp\biggl\{ \frac{\beta}{2}\cdot\frac
{1}{N}{
\sigma}'W_N{\sigma} \biggr\} \dto\exp\Biggl\{
\frac{\beta}{2} \sum_{i=1}^\infty
\lambda_i(W) (\xi_i-1) \Biggr\}.
\end{align}
If the LHS in \eqref{uniform_integrable} is uniformly integrable, then
%
%eA.4 #&#
%
\begin{eqnarray}
\lim_{N\rightarrow\infty}\E\exp\biggl\{ \frac{\beta}{2}\cdot
\frac{1}{N} {\sigma}'W_N{\sigma} \biggr
\}&=&\E\exp\Biggl\{ \frac{\beta}{2} \sum_{i=1}^\infty
\lambda_i(W) (\xi_i-1) \Biggr\}
\nonumber
\\[-8pt]
\\[-8pt]
\nonumber
&=&\prod_{i=1}^\infty\frac{e^{-\frac{1}{2}\beta\lambda
_i(W)}}{\sqrt{1-\beta\lambda_i(W) }},
\end{eqnarray}
where the last equality uses Lemma~\ref{lem:exists}. The proof of part
(a) then follows on taking log of both sides of the above equality.

It remains to show that the LHS in \eqref{uniform_integrable} is
uniformly integrable, that is,
%
%eA.5 #&#
%
\begin{align}
\limsup_{N\rightarrow\infty}\log\E_{0}\exp\biggl\{
\frac{\beta
+\delta}{2}\cdot\frac{1}{N} \sigma'W_N
\sigma\biggr\}=\lim_{N\rightarrow\infty}F_N(\beta+\delta)<\infty,
\label{unint}
\end{align}
for some $\delta>0$. To this end, note that if $0<\beta<1/\llVert
W\rrVert $,
there exists $\delta>0$ such that $\gamma:=\beta+\delta<1/\llVert
W\rrVert $.
Now, using \eqref{eq:upper_bound} and the fact $\sum_{i=1}^N\lambda
_i(G_N)=0$, we have
%
%eA.6 #&#
%
\begin{align}
\label{uint_2} F_N(\gamma) \le& \sum_{i=1}^N
\biggl\{-\frac{1}{2}\log\biggl(1- \frac{\gamma\lambda_i(W_N)}{N}
\biggr)-
\frac{\gamma}{2} \cdot\frac{\lambda_i(W_N)}{N} \biggr\}.
\end{align}
Since $W_N \Rightarrow W$ in the cut metric, $\lim_{N\rightarrow
\infty}\frac{\gamma\lambda_i(W_N)}{N}=\gamma\llVert W\rrVert
<1$, and so
there exists $\varepsilon>0$ such that for all $N$ large enough $\frac
{\gamma\lambda_i(W_N)}{N}\leq1-\varepsilon$. For $x\le
1-\varepsilon$ there exists $M=M(\varepsilon)$ such that $-\log
(1-x)-x\le Mx^2$. Using this the RHS of \eqref{uint_2} can be bounded
by $\frac{M\gamma^2\sum_{i=1}^N\lambda_i^2(W_N)}{N^2}$ which
converges to $M\gamma^2\llVert W\rrVert _2^2=M\gamma^2$, as
$N\rightarrow\infty
$. This proves (\ref{unint}) and completes the proof of the proposition.
\end{appendix}

\section*{Acknowledgements}
The authors thank Sourav Chatterjee, Persi Diaconis, and
Qingyuan Zhao for helpful discussions. The authors also thank the
anonymous referees for their careful comments which improved the
presentation of the paper.

% imsref loaded by akundreckaite, 2016-09-12 14:03:24
%

%\begin{appendix}
%\section{}
%\end{appendix}

%\section*{Acknowledgements}

%\begin{supplement}%[id=suppA]
%\slink[doi]{10.3150/00-BEJXXXXSUPP} %[doi,text={...}] - jei reikia
%suskaldyti doi
%\sfilename{BEJ000\_supp.pdf}
%\end{supplement}

%\begin{thebibliography}{00}
%\bibitem[\protect\citeauthoryear{}{()}]{r1}
%\bibitem{r1}
%\end{thebibliography}

%\printhistory
\end{document}